\documentclass[12pt,amstex]{amsart}

\usepackage{mathptmx}
\usepackage{mathrsfs}

\usepackage{verbatim}
\usepackage{url}
\usepackage[all]{xy}
\usepackage{color}

\usepackage[colorlinks=true,citecolor=blue]{hyperref}

\usepackage{stmaryrd}
\usepackage{epsfig}
\usepackage{amsmath}
\usepackage{amssymb}

\usepackage{amscd}
\usepackage{graphicx}
\usepackage{pstricks}

\theoremstyle{plain}
\newtheorem{thm}{Theorem}[section]
\newtheorem{lem}[thm]{Lemma}

\newtheorem{prop}[thm]{Proposition}
\newtheorem{ques}{Question}
\newtheorem{conj}{Conjecture}
\newtheorem{cor}[thm]{Corollary}
\theoremstyle{definition}
\newtheorem{de}[thm]{Definition}

\newtheorem{rem}[thm]{Remark}

\def \N {\mathbb N}

\def \Z {\mathbb Z}
\def \R {\mathbb R}

\def \O {\mathcal{O}}

\def \G {\mathcal{G}}

\def \X {\mathcal{X}}

\def \a {\alpha }

\def \ep {\epsilon}
\def \d {\delta}
\def \D {\Delta}

\begin{document}

\title{Minimal systems with finitely many ergodic measures}

\author{Wen Huang}

\author{Zhengxing Lian}
\author{Song Shao}
\author{Xiangdong Ye}

\address{Wu Wen-Tsun Key Laboratory of Mathematics, USTC, Chinese Academy of Sciences and
Department of Mathematics, University of Science and Technology of China,
Hefei, Anhui, 230026, P.R. China.}

\email{wenh@mail.ustc.edu.cn}
\email{lianzx@mail.ustc.edu.cn}
\email{songshao@ustc.edu.cn}
\email{yexd@ustc.edu.cn}

\subjclass[2010]{Primary: 37B05; 54H20}

\thanks{This research is supported by NNSF of China (11971455, 11731003, 11571335,  11431012).}

\date{}

\begin{abstract}
In this paper it is proved that if a minimal system has the property that its sequence entropy
is uniformly bounded for all sequences, then it has only finitely
many ergodic measures and is an almost finite to one extension of its maximal equicontinuous factor.
This result is obtained as an application of a general criteria which states that  if a minimal
system is an almost finite to one extension of its maximal
equicontinuous factor and has no infinite independent sets of length  $k$  for some $k\ge 2$, then it has only finitely
many ergodic measures.


\end{abstract}

\maketitle





\section{An introduction and the summary of the results}

In this section we first give the background of the study and then state the main results of the paper.

\subsection{The background}\
\medskip

A {\em topological dynamical system} is a pair $(X,
T)$, where $X$ is a compact metric space and $T : X \rightarrow  X$
is a homeomorphism. Denote by $M(X)$ the set of all Borel
probability measures on $X$. Let $M_T(X)=\{\mu\in M(X):
T_*\mu=\mu\circ T^{-1}=\mu\}$ be the set of all $T$-invariant
Borel measures of $X$. With the weak$^*$-topology, $M_T(X)$ is a compact convex space. By Krylov-Bogolioubov
theorem $M_T(X)\neq \emptyset$.  Denote by $M^{erg}_T(X)$ the set of ergodic measures of $(X,T)$, then $M^{erg}_T(X)$
is the set of extreme points of $M_T(X)$ and one can use the Choquet representation theorem to express each member
of $M_T(X)$ in terms of the ergodic members of $M_T(X)$. That is, for each $\mu\in M_T(X)$ there is a unique measure
$\tau$ on the Borel subsets of the compact space $M_T(X)$ such that $\tau(M^{erg}_T(X))=1$ and $\mu=\int_{M^{erg}_T(X)}m d\tau (m)$,
which is called the {\em ergodic decomposition} of $\mu$.

\medskip

Usually, the set $M^{erg}_T(X)$ may be very big, and thus it is interesting to consider the case when $M^{erg}_T(X)$ is small.
The extreme case is that $M^{erg}_T(X)$ consists of only one member, and in this case $(X,T)$ is said to be {\em uniquely ergodic}.
Uniquely ergodic systems are common, have lots of very nice properties and are very important in the study of dynamical systems.
For example, the well known
Jewett-Krieger's theorem asserts that every ergodic system is measurably isomorphic to a uniquely ergodic topological
system \cite{Jewett, Krieger}. The systems with only finitely many ergodic measures are also very common. For example, if $(X,T)$
is uniquely ergodic with a unique measure $\mu$, 
then  $(X,T^n)$ ($n\in\Z, n\not= 0$) has only finitely many ergodic measures.

\medskip
While there are lots of criteria for the unique ergodicity of a system, there are
very few conditions under which a system may have only finitely many ergodic measures. We now
state such a condition in \cite{Bo85} obtained for minimal subshifts by Boshernitzan.
Let $(\Omega, \sigma)$ be a minimal subshift over a finite alphabet.
Denote by $P(n)$ be the number of different $n$-blocks which appear in any $\omega\in \Omega$.
In \cite{Bo85}, it is shown that if there is some $r\ge 2$
such that $\liminf_{n\to+\infty} \left( P(n)-rn\right)=-\infty$, then $|M^{erg}_\sigma(\Omega)|\le r-1$; and if
$\liminf_{n\to+\infty} \frac{P(n)}{n}=\a<\infty$, then $|M^{erg}_\sigma(\Omega)|\le \max\{[\a],1\}$, where $[\a]$ denotes
the integer part of $\a$. Also it is shown that if $\limsup_{n\to+\infty} \frac{P(n)}{n}< 3$, then the system is uniquely ergodic.
We note that recently, the result was extended to a subshift of linear growth without the assumption of
minimality by Cyr and Kra \cite{CK-19}. It was proved that a subshift of linear growth has finitely many nonatomic
ergodic measures and thus has at most countably many ergodic measures (with no requirement
that the measures are nonatomic).

\medskip

In this paper we will give some other conditions when a system may have only finitely many ergodic measures.
To look for such conditions it is natural to consider minimal systems which are close to equicontinuous ones.
The class of systems we study in this paper is the collection of minimal systems having no $k$ infinite independent sets for some $k\ge 2$,
or no $k$ tuples with arbitrarily long finite independent sets for some $k\ge 2$.

\medskip
We remark that the case when $k=2$, namely the tame or null systems
were extensively studied in the literature.
A null topological dynamical system was defined by using the notion of sequence entropy.
Sequence entropy for a measure was introduced as an isomorphism
invariant by Kushnirenko \cite{Ku}, who used it to distinguish
between transformations with the same entropy and spectral
invariant. In the same paper, it was also shown that an invertible measure preserving
transformation has discrete spectrum if and only if it is null (the sequence entropy of the system is zero
for any sequence). Let
$(X,\mathcal{X}, \mu, T)$ be an ergodic system. Then the supremum of all sequence
entropies of $T$ is either $\log k$ for some $k\in \N$ or infinite \cite{Hu}.
The topological sequence entropy was introduced by Goodman in \cite{Go}.
Also, for a topological dynamical system $(X,T)$, the supremum of all topological sequence
entropies (denoted by $h_\infty(X, T)$ or $h^*(T)$) is either $\log k$ for some $k\in \N$ or infinite \cite{HY}.
It was shown \cite{HLSY} that if a minimal topological dynamical system is null
(the topological sequence entropy is zero for any sequence),
then it is uniquely ergodic, has discrete spectrum and is an almost one to one extension of its maximal equicontinuous factor.

\medskip

The concept of tameness was introduced by K\"{o}hler in \cite{Ko}. Here we follow the definition of
Glasner \cite{G06}. A topological dynamical system $(X,T)$ is said to be {\em tame} if its
enveloping semigroup is separable and Fr\'{e}chet, and it is said to be {\em non-tame} otherwise. It is known that a minima null system is tame \cite{KL}. A structure theorem for minimal tame systems has been established in \cite{H, KL, G07, G18}, i.e., a minimal tame system it is uniquely ergodic, has discrete spectrum and is an almost one
to one extension of its maximal equicontinuous factor. Recently, a striking result proved by Fuhrmann, Glasner, J\"{a}ger and Oertel solved a long open question, i.e.
the authors showed that a minimal tame system is regular \cite{FGJO}.

\medskip

In the sequel we will state the main results and some open
questions. From now on we will focus on topological dynamical systems under general group actions. We start by recalling some notions.







\subsection{Topological transformation groups}\
\medskip

A {\em topological dynamical system} (t.d.s for short) is a triple
$\X=(X, T, \Pi)$, where $X$ is a compact Hausdorff space, $T$ is a
Hausdorff topological group and $\Pi: T\times X\rightarrow X$ is a
continuous map such that $\Pi(e,x)=x$ and
$\Pi(s,\Pi(t,x))=\Pi(st,x)$, where $e$ is the unit of $T$, $s,t\in T$ and $x\in X$. We shall fix $T$ and suppress the
action symbol. Note that in the literatures, $\X$ is also called a {\em
topological transformation group} or a {\em flow}.

To be simple, we always assume that $T$ is infinite countable and discrete, unless we state explicitly in some places.
Moreover, we always assume that $X$ is a compact metric space with metric $d(\cdot, \cdot)$.

\medskip

Let $(X,T)$ be a t.d.s. and $x\in X$, then $\O(x,T)=\{tx: t\in T\}$ denotes the
{\em orbit} of $x$, which is also denoted by $T x$. We usually denote the closure of $\O(x,T)$ by $\overline{\O}(x,T)$, or $\overline{Tx}$. A subset
$A\subseteq X$ is called {\em invariant} if $t a\subseteq A$ for all
$a\in A$ and $t\in T$. When $Y\subseteq X$ is a closed and
$T$-invariant subset of the system $(X, T)$ we say that the system
$(Y, T)$ is a {\em subsystem} of $(X, T)$. If $(X, T)$ and $(Y,
T)$ are two t.d.s. their {\em product system} is the
system $(X \times Y, T)$, where $t(x, y) = (tx, ty)$ for any $t\in T$ and $x,y\in X$.

\medskip

A t.d.s. $(X,T)$ is called {\em minimal} if $X$ contains no proper non-empty
closed invariant subsets. 
It is easy to verify that a t.d.s. is
minimal if and only if every orbit is dense.

\medskip

A {\it factor map} $\pi: X\rightarrow Y$ between the t.d.s. $(X,T)$
and $(Y, T)$ is a continuous onto map which intertwines the
actions; we say that $(Y, T)$ is a {\it factor} of $(X,T)$ and
that $(X,T)$ is an {\it extension} of $(Y,S)$.
The systems are said to be {\it isomorphic} if $\pi$ is bijective.

\medskip

A t.d.s. $(X,T)$ is {\em equicontinuous} if for any $\ep>0$, there is a $\d>0$ such that whenever $x,y\in X$ with  $d(x,y)<\d$,
then $d(tx,ty)<\ep$ for all $t\in T$. Let $(X,T)$ be a t.d.s. There is a smallest invariant equivalence relation
$S_{eq}$ such that the quotient system $(X/S_{eq},T)$ is equicontinuous \cite{EG60}. The equivalence relation
$S_{eq}$ is called the {\em equicontinuous structure relation} and the factor $(X_{eq}=X/S_{eq}, T)$ is called
the {\em maximal equicontinuous factor} of $(X,T)$.

\medskip

Let $\pi:(X,T)\rightarrow (Y,T)$ be an extension of t.d.s.
We call that $\pi$ is {\em finite to one} if each fiber is finite, and {\em almost finite to one} if
there is a residual set $Y_0\subset Y$ such that for each $y\in Y_0$, the fiber of $y$ is finite.
We note that if $(Y,T)$ is minimal this is equivalent to say that there is a finite fiber.
If there is some $N\in \N$ and a dense $G_\d$ set $X_0$ of $X$ such that for each $x\in X_0$, the cardinality of the
fiber $\pi^{-1}(\pi(x))$ is $N$, then we also call $\pi$ is almost $N$ to 1. If a t.d.s. $(X,T)$ is minimal and is an almost one to one
extension of some equicontinuous system then we call it an {\em almost automorphic} system.

\subsection{A general criteria}\
\medskip


Let $(X,T)$ be a t.d.s. and $k\in \N$. We say that $(X,T)$ has no $k$-IT-tuple if for any tuple of
closed non-empty disjoint subsets $U_1$, $U_2$, $\ldots$, $U_{k}$ of $X$ there is no an infinite independence
set for them, i.e. for any infinite set $S\subseteq T$, there is some $a\in \{1,2,\ldots,k\}^S$ such that
      $$\bigcap_{t\in S}t^{-1}U_{a_t}=\emptyset.$$


The following is one of the main results of the paper.

\medskip
\noindent {\bf Theorem A}: 
{\it Let $(X,T)$ be a minimal system and let $T$ be an infinite countable discrete amenable group. If
$\pi: (X,T)\rightarrow (X_{eq},T)$ is almost finite to one, and there is some integer $k\ge 2 $ such that $(X,T)$  has no $k$-IT-tuples,
then $(X,T)$ has only finitely many ergodic measures.}

\medskip
Let us give a brief description of the idea of the proof. To prove Theorem A first we show that there is $N\in\N$ such that the set
$\{y\in X_{eq}: |\pi^{-1}(y)|=N\}$ is residual in $X_{eq}$ (Proposition \ref{almost N-1}). Then using the hyperspace technique,
we lift $\pi$ to an open $N$ to one map $\pi'$ through almost one to one extensions as follows
\begin{equation*}
  \xymatrix{
  X \ar[d]_{\pi}
                & X' \ar[d]^{\pi'} \ar[l]_{\sigma} \\
  X_{eq}
                & Y'   \ar[l]_{\tau}          }
\end{equation*}
where $\sigma$ and $\tau$ are almost one-to-one extensions. We can show that the length of any IT-tuple of $Y'$ is bounded by
$(k-1)^N$. At the same time the number of the ergodic measures of $Y'$ is bounded by the same number
 (Proposition \ref{key-prop-numberofIT}). Since $\pi'$ is $N$ to one, we conclude that the number of the
ergodic measures of $X'$ is bounded by $N(k-1)^N$, and so does $X$, ending the proof.
We mention that to obtain Proposition \ref{key-prop-numberofIT}
we need to use some result (Proposition \ref{independent sets}) which is a generalization of the previous one
obtained by Fuhrmann, Glasner, J\"{a}ger and Oertel \cite{FGJO}.

\medskip

We have the following remarks:

\begin{enumerate}
  \item In \cite[Subsection 5.3]{FGJO}, it was shown that there is a minimal system which is an at most
  two-to-one and almost one to one extension of its maximal equicontinuous factor, 
  but exhibits two distinct ergodic invariant measures. This system has 2-IT-tuples but no 3-IT-tuples.
  \item In Theorem A 
  we can only prove $(X,T)$ has only finitely many ergodic measures,
  and we are not able to get an upper bound only depending on $k$.
  We mention that if $\pi$ is almost one to one, then there are at most $k-1$ ergodic measures,
  see Proposition \ref{key-prop-numberofIT}.

  \item In fact, if $T$ is abelian, we can show that $(X,T)$ is an almost $N$ to one
  extension with $N\le k-1$. The same proof of Remark \ref{N'N} can be applied here.
\end{enumerate}



By the proof of Theorem A, 
we have the following corollary.

\medskip
\noindent {\bf Corollary B}:
{\it Let $(X,T)$ be a minimal system  and let $T$ be an infinite countable discrete amenable group. If
$\pi: (X,T)\rightarrow (X_{eq},T)$ is finite to one, then $(X,T)$ has only finitely many ergodic measures.}


\medskip

We believe that in Theorem A, 
the condition that $(X,T)$ is an almost finite to one extension of its maximal equicontinuous factor
is superfluous. To be precise, we have the following conjecture:

\begin{conj}
Let $(X,T)$ be a minimal system with $T$ infinite countable discrete. If there is some integer $k\ge 2 $ such that $(X,T)$  has no $k$-IT-tuples,
then $(X,T)$ is an almost finite to one  extension of its maximal equicontinuous factor. 
\end{conj}

In fact this conjecture is closely related to the following one

\begin{conj} Let $(X,T)$ be a minimal t.d.s. with $T$  infinite countable discrete amenable and $\pi:X\rightarrow X_{eq}$ be the factor map
to the maximal equicontinuous factor of $(X,T)$. If $\pi$ is proximal and not almost one to one,
then for each $k\ge 2$, there is a $k$-IT tuple.
\end{conj}



\subsection{Bounded minimal systems}\

\medskip

We may use sequence entropy to give a very succinct criteria for systems with only finitely many ergodic measures.

\medskip


In this article we focus on  minimal t.d.s. with
bounded topological sequence entropy. That is, there is a positive
real number $M$ such that for each increasing sequence of positive
integer numbers, the topological entropy of the system along this
sequence is bounded by $M$. This class of systems was studied in \cite{HLSY} and \cite{MS}.
As another main result, a structure theorem for a bounded minimal system is obtained. That is,

\medskip
\noindent {\bf Theorem C}: 
{\it If $(X,T)$ is a bounded minimal system with $T$ abelian, then it is an almost finite to one extension of
its maximal equicontinuous factor, and it has finitely many ergodic measures.}

\medskip
Let us give a brief account of the idea of the proof. Maass and Shao \cite{MS} proved that
under the assumption of the theorem we have
\[
\begin{CD}
X @<{\sigma'}<< X'\\
@VV{\pi}V      @VV{\pi'}V\\
X_{eq} @<{\tau' }<< Y'
\end{CD}
\]
where $X_{eq}$ is the maximal equicontinuous factor of $X$, $\sigma'$ and $\tau'$ are proximal extensions, and $\pi'$ is an
$N$ to one extension for some $N\in \N$. So it is left to prove that in fact $\sigma'$ and $\tau'$ are almost one to one according to Theorem A.
It is done by showing that if $(X,T)$ is minimal and $\pi:X\rightarrow X_{eq}$ is proximal and not almost one to one,
then $h_\infty(X,T)=\infty$, i.e. $(X,T)$ is not bounded,
see Proposition \ref{proximal-extension}.

\medskip


It is easy to see that by the Rohlin's skew-product theorem, and the fact that for each invariant measure
the sequence entropy is bounded by the topological sequence entropy, each ergodic measure from Theorem C
can be expressed as skew product of a Kronecker system with a periodic system. It is an
interesting question to understand the finer structure of the ergodic measures. In the following we
will state a question on the uniqueness of the measures.

\medskip

Let us see an example first. Let $\tau$ be the substitution $\tau(0)=01$ and $\tau(1)=10$. By
concatenating, this map can be defined on any finite word
$w=w_0\ldots w_{l-1}$ in $\{0,1\}$:
$\tau(w)=\tau(w_0)\ldots\tau(w_{l-1})$. For any $n\geq 2$ define
$\tau^n(w)=\tau(\tau^{n-1}(w))$. Finally define $X\subseteq
\{0,1\}^\Z$ to be the set of biinfinite binary sequences $x$ in
$X$ such that any finite word in $x$ is a subword of $\tau^n(0)$
for some $n \in \N$. The t.d.s. $(X,T)$, where $T$ is the
left shift map is called a Morse system. It is well known that it is
minimal and has the following structure: $\pi_1: X\to Y$ and
$\pi_2: Y \to X_{eq}$ where $\pi_1$ is a $2$-to-one distal
extension and $\pi_2$ is an asymptotic extension (so almost one to
one) \cite{Vr}. Thus $\pi=\pi_2\circ\pi_1: X\rightarrow X_{eq}$ is almost 2-to-1 and
$h_\infty(X,T)=\log 2$. The Morse system is uniquely ergodic, and $\pi$ is regular almost 2-to-1.

\medskip

Inspired by the structure of the Morse system, we have the following question.
First we give a definition. Let $(X,T)$ be a minimal system and $\pi: X\rightarrow X_{eq}$ be the extension of its maximal equicontinuous
factor. Let $\pi$ be an almost $N$  to 1 and let $Y_N=\{y\in X_{eq}: |\pi^{-1}(y)|=N\}$. If $m(Y_N)=1$,
then we say that $\pi$ is a {\em regular almost $N$  to 1} map, where $m$ is the Haar measure on $X_{eq}$.
In \cite{FGJO}, it was shown that a minimal tame system is regular. 
\begin{ques}
Let $(X,T)$ be a minimal system with $T$ abelian. Assume that $\pi: (X,T)\rightarrow (X_{eq},T)$ is almost
$N$-to-1 and $h_\infty (X,T)=\log N$. Is it true that $(X,T)$ has a structure as  $X\rightarrow Y\rightarrow X_{eq}$,
where $(Y,T)$ is the maximal null factor and $X\rightarrow Y$ is open $N$ to one? Moreover,
is it true that $(X,T)$ is uniquely ergodic? Is it true that $\pi$ is regular almost $N$ to 1?
\end{ques}

We remark that when $N=1$, it is true, that is a minimal null system is uniquely ergodic \cite{HLSY} and regular \cite{FGJO}.



\subsection{The Sarnak conjecture}\
\medskip

Let $X$ be a compact metric space and $f: X\rightarrow X$ a homeomorphism.
We say that a topological dynamical system $(X,f)$ satisfies the {\em Sarnak conjecture} if for
every continuous function $g$ on $X$ and every $x\in X$, the Ces\`{a}ro averages
$$\frac{1}{N}\sum_{n=1}^Ng(f^nx){\bf \mu}(n)$$
tend to $0$ as $N\to \infty$.
We say that a topological dynamical system $(X,f)$ satisfies the {\em logarithmic Sarnak conjecture}
if for every continuous function $g$ on $X$ and every $x\in X$, the logarithmic averages
$$\frac{1}{\log N}\sum_{n=1}^N\frac{g(f^nx){\bf \mu}(n)}{n}$$
tend to $0$ as $N\to \infty$. Note that the Sarnak conjecture for a system implies
the logarithmic Sarnak conjecture for the same system. Frantzikinakis and Host \cite{FH}
showed that if $(X,f)$ is a t.d.s. with zero topological entropy
and has countably many ergodic invariant measures. Then $(X,f)$ satisfies the logarithmic Sarnak conjecture.
Thus, together with the mentioned result of Frantzikinakis and Host \cite{FH} and Theorem C, we have

\medskip
\noindent {\bf Corollary D}:
{\it Any bounded minimal system satisfies the logarithmic Sarnak conjecture.}

\medskip

To end the section we ask
\begin{ques}
Does any bounded minimal system satisfy the Sarnak conjecture?
\end{ques}


\bigskip

We organize the paper as follows. 
After introducing necessary notations and results in Section \ref{sec-pre}, we prove the main results in
Section \ref{sec-proof-main} and Section \ref{sec-SequenceEntropy}. In the appendix we will give
the proof of Proposition~\ref{independent sets}.

\bigskip

\noindent{\bf Acknowledgments:} We would like to thank Hanfeng Li, Jian Li and Tao Yu for their very useful comments.

\section{Preliminaries}\label{sec-pre}

In the article, integers, nonnegative integers and natural numbers
are denoted by $\Z$, $\Z_+$ and $\N$ respectively.

\subsection{The Ellis semgroup}\
\medskip

Given a t.d.s. $(X,T)$, the {\em Ellis semigroup} $E(X,T)$ associated to
$(X, T )$ is defined as the closure of $\{x\mapsto tx: t\in T\}\subset X^X$ in the product topology,
where the semi-group operation is given by the composition \cite{E601}. On $E(X,T)$, we may consider the
$T$-action given by $E(X,T)\rightarrow E(X,T), p\mapsto tp$. A well known result by Ellis says that
$(X,T)$ is equicontinuous if and only if $E(X,T)$ is a topological group \cite[Theorem 3, Chapter 3]{Au88}.

\begin{thm}\cite[Theorem 6, Chapter 3]{Au88}\label{Thm-Ellis}
Let $(X,T)$ be an equicontinuous minimal t.d.s., and let $x\in X$.
Let $\Gamma= \Gamma_x = \{p\in E(X,T):px=x\}$. Then $\Gamma$ is a closed subgroup
of $E(X,T)$, $T$ acts on the space of right cosets $\{p\Gamma : p\in E(X,T)\}$, by
$t(p\Gamma)=(tp)\Gamma, (t\in T)$, and the system $(E(X,T)/\Gamma,T)$ is isomorphic with $(X,T)$.

If $T$ is abelian, then $\Gamma =\{e\} $ and $(X,T)$ is isomorphic to $(E(X,T),T)$.
\end{thm}

\subsection{Independence and tameness}\
\medskip

One may use independence sets to give an equivalent definition of tameness.

\begin{de}
Let $(X,T)$ be a t.d.s. For a tuple $A=(A_1,A_2,\ldots, A_k)$ of subsets of $X$,
we say that a set $J\subseteq T$ is an {\em independence set} for $A$ if for
every nonempty finite subset $I\subseteq J$ and function
$\sigma: I\rightarrow \{1,2,\ldots, k\}$ we have $$\bigcap_{s\in I} s^{-1} A_{\sigma(s)}\neq \emptyset.$$
\end{de}

\begin{de}\cite{KL}
Let $(X,T)$ be a t.d.s. and $n\ge 2$.
We call a tuple $x=(x_1,\ldots,x_n)\in X^n$ an {\em IT-tuple} (or an {\em IT-pair} if $n = 2$)
if for any product neighbourhood $U_1\times U_2\times \ldots \times U_n$ of $x$ in $X^n$
the tuple $(U_1,U_2,\ldots, U_n)$ has an infinite independence sets. We denote the set of IT-tuples of length $n$ by ${\rm IT}_n (X, T)$.
\end{de}

The diagonal of $X^n$ is defined by
$$\Delta_n(X)=\{(x,\ldots,x)\in X^n: x\in X\}$$ and put
$$\Delta^{(n)}(X)=\{(x_1,\ldots,x_n) \in X^n: \text{ for some
$i\neq j$ }, x_i=x_j  \}.$$ When $n=2$ one writes
$\Delta(X)=\Delta_2(X)=\Delta^{(2)}(X)$.

\begin{prop}\cite[Proposition 6.4]{KL}\label{factor and IT tuple}
Let $(X,T)$ be a t.d.s. and $n\geq 2$.
\begin{enumerate}
  \item Let $(A_1,\ldots, A_n )$ be a tuple of non-empty closed subsets of $X$ which has infinite
  independence sets. Then there exists an IT-tuple $(x_1,\ldots,x_n)$ with $x_j\in A_j$ for all $1\le j\le n$.
  \item  ${\rm IT}_2(X, T) \setminus \D_2(X)$ is nonempty if and only if $(X, T)$ is non-tame.
  \item ${\rm IT}_n (X, T)$ is a closed T-invariant subset of $X^n$.
  \item Let $\pi: (X,T)\rightarrow (Y,T)$ be a factor map. Then $\pi^{(n)}( {\rm IT}_n(X,T))={\rm IT}_n(Y,T)$,
  where $\pi^{(n)}: X^n\rightarrow Y^n$ defined by $\pi^{(n)}(x_1,x_2,\ldots, x_n)=(\pi(x_1),\pi(x_2),\ldots, \pi(x_n))$.
  \item Suppose that $Z$ is a closed $T$-invariant subset of $X$. Then ${\rm IT}_n (Z, T)\subseteq  {\rm IT}_n (X, T)$.
\end{enumerate}
\end{prop}

\subsection{Sequence entropy (maximal pattern entropy) and independence}\
\medskip


Let $(X, T)$ be t.d.s. Consider an infinite sequence $A= \{  t_i \}_{i=1}^\infty\subset T$ and a finite open cover $\mathcal{U}$ of $X$. The {\it
topological sequence entropy} of $\mathcal{U}$ with respect to
$(X,T)$ along $A$ is
$$h_A(T,\mathcal{U})= \limsup_{n\rightarrow \infty}\frac{1}{n}\log
{N}(\bigvee_{i=1}^{n}t_i^{-1}\mathcal{U}),$$ where
${N}(\bigvee_{i=1}^{n}t_i^{-1}\mathcal{U})$ is the minimal
cardinality among all cardinalities of subcovers of
$\bigvee_{i=1}^{n}t_i^{-1}\mathcal{U}$. Recall that for open
covers $\mathcal{U}$ and $\mathcal{V}$ of $X$, $\mathcal{U}
\bigvee \mathcal{V}=\{ U \cap V : U \in \mathcal{U}, V \in
\mathcal{V} \}$.

The {\it topological sequence entropy of $(X,T)$} along $A$ is
$$h_A(X,T)= \sup_{\mathcal{U}} h_A(T,\mathcal{U}),$$
where the
supremum is taken over all finite open covers of $X$.


Finally the {\em sequence entropy} of $(X,T)$ is defined
by
$$h_{\infty}(X,T)=\sup h_A(X,T),$$
where the supremum ranges over all infinite sequences of $T$. The sequence entropy of a
system is also called {\em the maximal pattern entropy} \cite{HY}.

\medskip

An important fact is as follows:
\begin{thm}\cite{HY}
Let $(X,T)$ be a t.d.s. Then
\begin{equation*}
  h_\infty(X,T)\in \{\log n: n\in \N\}\cup \{\infty\}.
\end{equation*}
\end{thm}

\begin{de}
Let $(X,T)$ be a t.d.s. $(X,T)$ is called
\begin{enumerate}
  \item {\em null} if $h_{\infty}(X,T)=0$;
  \item {\em bounded} if $h_{\infty}(X,T)<\infty$;
  \item {\em
unbounded} if $h_{\infty}(X,T)=\infty$.
\end{enumerate}
\end{de}

By an {\it admissible cover} $\mathcal{U}$ of $X$ one means that
$\mathcal{U}$ is finite and if $\mathcal{U}=\{U_1,\ldots,U_n\}$
then $( \bigcup _{j\neq i} U_j)^c$ has nonempty interior for each
$i\in \{1,\ldots,n\}$. Let $(x_1,\ldots,x_n) \in X^{n}$ and
$\mathcal{U}=\{U_1,\ldots,U_n \}$ be a finite cover of $X$. One
says $\mathcal{U}$ is an {\em admissible cover} with respect to
$(x_1,\ldots,x_n)$ if for each $U_i$,  $i\in \{1,\ldots,n\}$,
there exists $j_i \in \{1,\ldots,n\}$ such that $x_{j_i}$ is not
in the closure of $U_{i}$.

\begin{de}
Let $(X,T)$ be a t.d.s. and $n\ge 2$. An $n$-tuple $(x_1,\ldots,x_n)
\in X^{n}\setminus \D_n(X)$ is a {\em sequence entropy $n$-tuple}
($n$-SET) if whenever $V_1,\ldots,V_n$ are closed mutually
disjoint neighborhoods of $x_1,\ldots,x_n$ respectively, there is
some infinite sequence $A\subset T$ such that the open
cover $\mathcal {U} =\{ V_1^c,\ldots, V_n^c\}$ has positive
sequence entropy along $A$, i.e. $h_{A}(T,\mathcal {U})>0$.
\end{de}

It is easy to see that an $n$-tuple $(x_1,\ldots,x_n) \in
X^{n}\setminus \D_n(X)$ is an $n$-SET if and only if for any
admissible open cover $\mathcal{U}$ with respect to
$(x_1,\ldots,x_n)$ one has $h_{A}(T,\mathcal {U})>0$ for some sequence $A\subset T$.

For $n\ge 2$ one denotes by $SE_n(X,T)$ the set of $n$-SET. In the
case $n=2$ one speaks about pairs instead of tuples and one writes
$SE(X,T)$. The proof of the following result is analogous to the
corresponding one in \cite{B2} (see \cite[Propositions 2, 3, 4 and 5]{B2}).

\begin{prop}\cite[Proposition 2.6.]{MS}\label{prop:basicseq}
 Let $(X,T)$ be a t.d.s. and $n\geq 2$.
\begin{enumerate}
    \item If $\mathcal {U}=\{ U_1,\ldots, U_n\}$ is an admissible
    open cover of $X$
with $h_{A}(T,\mathcal {U})>0$ for some sequence $A\subset T$, then for each $i \in \{1,\ldots,n\}$ there
exists $x_i \in U_i^c$ such that $(x_1,\ldots,x_n)$ is an n-SET.

    \item $SE_n(X,T)\cup \Delta_n(X)$ is a nonempty closed
$T$-invariant subset of $X^{n}$.

    \item Let $\pi : (X,T)\rightarrow (Y,T)$ be a factor map. Then $\pi^{(n)}(SE_n(X,T))=SE_n(Y,T)$.

\item Let $W$ be a closed $T$-invariant subset of $(X,T)$.
$SE_n(W,T)\subseteq SE_n(X,T)$.
\end{enumerate}
\end{prop}

By Proposition \ref{prop:basicseq}, a system $(X,T)$
is null if and only if $SE(X,T)=\emptyset$.

\medskip



One may use independence to characterize sequence entropy tuples.

\begin{de}\cite{KL}
Let $(X,T)$ be a t.d.s. and $n\geq 2$.
We call a tuple $x=(x_1,\ldots,x_n)\in X^n$ an {\em IN-tuple} (or an {\em IN-pair} if $n = 2$)
if for any product neighbourhood $U_1\times U_2\times \ldots \times U_n$ of $x$
the tuple $(U_1,U_2,\ldots, U_n)$ has arbitrarily large finite independence sets. We denote the
set of IN-tuples of length $n$ by ${\rm IN}_n (X, T)$.
\end{de}

Note that for ${\rm IN}_n(X,T)$, we have the similar properties listed in Proposition \ref{factor and IT tuple}. The following result explain the relations between $IN$-tuples and sequence entropy tuples.



\begin{thm}\cite[Theorem 5.9]{KL}
Let $(X,T)$ be a t.d.s. and let $n\in \N$ with $n\ge 2$. Then $$SE_n(X,T)\cup \Delta_n(X)={\rm IN}_n(X,T).$$
\end{thm}


The following lemma is proved in \cite{HY}.

\begin{lem}\label{exist of se}
Let $(X,T)$ be a t.d.s., and let $n\in \N$ with $n\ge 2$. Then $h_\infty(X,T)\ge \log n$ if and only if $SE_n(X,T)\setminus \D^{(n)}(X)\neq \emptyset$.
\end{lem}

We will also use the following result in the sequel.
\begin{thm}\cite[Theorems 3.8, 3.9]{MS}\label{MS-lemma}
Let $(X,T)$ be a minimal system with $T$ abelian and $\pi: X\rightarrow X_{eq}$ be the extension of its maximal
equicontinuous factor. Let $x_1,x_2,\ldots, x_n\in X$ such that $\pi(x_1)=\ldots=\pi(x_n)$.
If $(x_1,\ldots,x_n)$ is minimal, then $(x_1,\ldots,x_n)\in SE_n(X,T)\cup \Delta_n(X)$.
\end{thm}

In fact, the same proof yields that $(x_1,\ldots,x_n)\in IT_n(X,T)$.

\subsection{Some facts about hyperspaces} \label{sub:ellis} \
\medskip

Let $X$ be a compact Hausdorff topological space. Let $2^X$ be the collection of nonempty closed subsets of $X$
endowed with the Hausdorff topology. A basis for this
topology on $2^X$ is given by
$$\langle U_1,\ldots,U_n \rangle=\{A\in 2^X: A\subseteq \bigcup_{i=1}^n U_i\
\text{and $A\cap U_i\neq \emptyset$ for every $i\in \{
1,\ldots,n\}$}\},$$ where each $U_i \subseteq X$ is open.
When $X$ is a metric space, then $2^X$ is also a metric space. Let $d$ be the metric on $X$,
then one may define a metric on $2^X$ as follows:
\begin{equation*}
\begin{split}
 d_H(A,B)& = \inf \{\ep>0: A\subset B_\ep[B], B\subset B_\ep[A]\}\\
 &= \max \{\max_{a\in A} d(a,B),\max_{b\in B} d(b,A)\},
\end{split}
\end{equation*}
where $d(x,A)=\inf_{y\in A} d(x,y)$ and $B_\ep [A]=\{x\in X: d(x, A)<\ep\}$.
The metric $d_H$ is called the {\em Hausdorff metric} of $2^X$.

Let $\{A_i\}_{i=1}^\infty$ be a sequence of subsets of $X$. Define
$$\liminf A_i=\{x\in X: \text{for any neighbourhood $U$ of $x$, $U\cap A_i\neq \emptyset$ for all but finitely many $i$}\};$$
$$\limsup A_i=\{x\in X: \text{for any neighbourhood $U$ of $x$, $U\cap A_i\neq \emptyset$ for infinitely many $i$}\};$$
We say that $\{A_i\}_{i=1}^\infty$ converges to $A$, denoted by $\lim_{i\to \infty} A_i=A$, if
$$\liminf A_i=\limsup A_i=A.$$

Let $X,Y$ be two compact Hausdorff topological spaces. Let $F: Y\rightarrow 2^X$ be a map and $y\in Y$.
 We say that $F$ is {\em upper semi-continuous (u.s.c.)} at $y$ if for any open set $U$ of $X$
 such that $F(y)\subset U$, then $\{y'\in Y: F(y')\subset U\}$ is a neighbourhood of $y$.
 If $F$ is u.s.c. at every point of $Y$, then we say that $F$ is u.s.c. It is easy to see
 that $F$ is u.s.c. at $y$ if and only if whenever $\lim y_i=y$, one has that
 $\limsup F(y_i)\subset F(y)$. If $f: X\rightarrow Y$ is a continuous map, then it is
 easy to verify that $$F=f^{-1}: Y\rightarrow 2^X, y\mapsto f^{-1}(y)$$ is u.s.c.

We say $F$ is {\em lower semi-continuous (l.s.c.)} at $y$ if for any open set $U$ of $X$
such that $F(y)\cap U\neq \emptyset$, then $\{y'\in Y: F(y')\cap U\neq \emptyset\}$ is a
neighbourhood of $y$. If $F$ is l.s.c. at every point of $Y$, then we say that $F$ is l.s.c.
It is easy to see that $F$ is l.s.c. at $y$ if and only if whenever $\lim y_i=y$, one has that $\liminf F(y_i)\supset F(y)$.

We have the following well known result, for a proof see \cite[p.70-71]{Kura2} and \cite[p.394]{Kura1}.
\begin{thm}\label{Kura}
Let $X,Y$ be compact metric spaces. If $F: Y\rightarrow 2^X$ is u.s.c. (or l.s.c.),
then the points of continuity of $F$ form a dense $G_\delta$ set in $Y$.
\end{thm}

\medskip

Let $(X,T)$ be a t.d.s. We can induce a system on $2^X$. The
action of $T$ on $2^X$ is given by $tA=\{ta:a\in A\}$ for each $t\in T$ and $A \in 2^X$. Then
$(2^X,T)$ is a t.d.s. and it is called the {\em hypersapce system}.

\subsection{Fundamental extensions}\
\medskip

Let $(X,T)$ be a t.d.s. Fix $(x,y)\in X^2$. It is a {\it proximal}
pair if $\inf_{t\in T} d(tx, ty)=0$; it is a {\it
distal} pair if it is not proximal. Denote by $P(X,T)$ and
$D(X,T)$ the sets of proximal and distal pairs of $(X,T)$
respectively. They are also called the proximal and distal
relations. A t.d.s. $(X,T)$ is {\it distal} if $D(X,T)=X^2 \setminus
\D(X)$. Any equicontinuous system is distal.

\medskip

Let $(X,T)$ and $(Y,S)$ be t.d.s. and let $\pi: X \to Y$ be a factor map.
One says that:
\begin{enumerate}
  \item $\pi$ is an {\it open} extension if it is open as a map;
  \item $\pi$ is a {\it semi-open} extension if the image of every nonempty open set of $X$ has nonempty interior;
  \item $\pi$ is a {\it proximal} extension if
$\pi(x_1)=\pi(x_2)$ implies $(x_1,x_2) \in P(X,T)$;
  \item $\pi$ is  a {\it distal} extension if $\pi(x_1)=\pi(x_2)$ and $x_1\neq x_2$ implies $(x_1,x_2) \in D(X,T)$;
\item $\pi$ is an {\it equicontinuous or isometric} extension if for any $\ep >0$ there exists $\d>0$
such that $\pi(x_1)=\pi(x_2)$ and $d(x_1,x_2)<\d$ imply $d(tx_1,tx_2)<\ep$ for any $t\in T$;
\item $\pi$ is a {\it group} extension if there exists a
compact Hausdorff topological group $K$ such that the following
conditions hold:
\begin{enumerate}
\item $K$ acts continuously on $X$ from the right: the
right action $X\times K\rightarrow X$, $(x,k)\mapsto xk$ is
continuous and $t(xk)=(tx)k$ for any $t\in T$ and $k\in K$;
\item the fibers of $\pi$ are the $K$-orbits in $X$:
$\pi^{-1}(\{\pi(x)\})=xK$ for any $x\in X$.
\end{enumerate}
\end{enumerate}

Note that a group extension is equicontinuous, and an equicontinous extension is distal.



\subsection{Almost finite to one extensions}\label{subsec-finite-one}\
\medskip

In this subsection we collect some known properties about finite to one extensions and almost finite to one extensions.


Let $\pi:(X,T)\rightarrow (Y,T)$ be an extension of t.d.s. Let $\pi^{-1}: Y\rightarrow 2^X, y\mapsto \pi^{-1}(y)$.
Then it is easy to verify that $\pi^{-1}$ is a u.s.c. map, and by Theorem \ref{Kura}, the set $Y_c$ of
continuous points of $\pi^{-1}$ is a dense $G_\d$ subset of $Y$.
Let $$\widetilde{Y}=\overline{\{\pi^{-1}(y): y\in Y\}}\ \text{and}\ Y'=\overline{\{\pi^{-1}(y): y\in Y_c\}},$$
where the closure is taken in $2^X$. It is obvious that $Y'\subseteq \widetilde{Y}\subseteq 2^X$. Note that
for each $A\in \widetilde{Y}$, there is some $y\in Y$ such that $A\subseteq \pi^{-1}(y)$, and hence $A\mapsto y$
define a map $\tau: \widetilde{Y}\rightarrow Y$. It is easy to verify that $\tau: (\widetilde{Y},T)\rightarrow (Y,T)$
is a factor map. Now we show that if $(Y,T)$ is minimal then $(Y',T)$ is a minimal t.d.s. and it is the
unique minimal subsystem in $(\widetilde{Y},T)$. To see this, let $(M,T)$ be a minimal subsystem of $(\widetilde{Y},T)$.
Since $(Y,T)$ is minimal, $\tau: M\rightarrow Y$ is surjective. Let $y_0\in Y_c$ and $A\in M$ with $\tau(A)=y_0$.
By the definition of $\tau$, $A\subseteq \pi^{-1}(y_0)$.
Since $A\in M\subseteq \widetilde{Y}$, there is some sequence $\{y_i\}_{i=1}^\infty\subseteq Y$ such that
$\pi^{-1}(y_i)\to A, i\to\infty$. As $A\subseteq \pi^{-1}(y_0)$, it follows that $y_i\to y_0, i\to\infty$.
By the fact $y_0\in Y_c$, we have that
$\pi^{-1}(y_i)\to \pi^{-1}(y_0)$, $i\to\infty$. Thus $A=\pi^{-1}(y_0)$. To sum up, we have showed that for
each $y_0\in Y_c$, $\pi^{-1}(y_0)$ is a minimal point of $(2^X,T)$ and
$$\{\pi^{-1}(y_0): y_0\in Y_c\}\subseteq M.$$
Thus $Y'=\overline{\{\pi^{-1}(y_0): y_0\in Y_c\}}\subseteq M$. Since $M$ is minimal, $Y'=M$. That is, $(Y',T)$
is the unique minimal subsystem in $(\widetilde{Y},T)$ and $\tau: Y'\rightarrow Y$ is an almost one to one
 extension. Note the this result was given by Veech in \cite{V70}, and see also \cite{AuG77, Shoenfeld, V} for generalizations.

Using this result we can give some equivalent conditions for almost finite to one extensions.

\begin{prop}\label{almost N-1}
Let $\pi:(X,T)\rightarrow (Y,T)$ be an extension with $(Y,T)$ being minimal.
The following statements are equivalent:
\begin{enumerate}
  \item $\pi$ is almost finite to one, i.e. some fiber is finite;
  \item There exists $N\in \N$ such that $Y_0=\{y\in Y: |\pi^{-1}(y)|=N\}$ is residual, i.e. it contains a dense $G_\d$ subset of $Y$;
  \item There exist $N\in \N$ and $y_0\in Y$ such that $|\pi^{-1}(y_0)|=N$ and $\pi^{-1}(y_0)$ is a minimal point of $(2^X, T)$.
\end{enumerate}
\end{prop}

\begin{proof}
Let $\pi^{-1}: Y\rightarrow 2^X$, $Y_c$ and $Y'$ etc. be defined as above.

(1) $\Rightarrow$ (2):
If $\pi$ is almost finite to one, then the set $Y_f:=\{ y_0\in Y:|\pi^{-1}(y_0)|<\infty\}$ is not empty.
Put $N:=\min_{y_0\in Y_f} |\pi^{-1}(y_0)|$. Then $N\in \mathbb{N}$. Take any $y\in Y_c$ and $y_0\in Y_f$. Since $(Y,T)$ is minimal, there is some sequence $\{t_i\}_{i=1}^\infty\subset T$ such that $\lim\limits_{i} t_iy_0=y$. Note that $|\pi^{-1}(t_iy_0)|=|\pi^{-1}(y_0)|$ for $i\in \mathbb{N}$  and $y\in Y_c$. It follows
that $|\pi^{-1}(y)|\le |\pi^{-1}(y_0)|<\infty$.
This implies that $Y_c\subset Y_f$ and
$$Y_c\subset Y_0=\{y\in Y: |\pi^{-1}(y)|=N\}.$$
Thus we have that $Y_0$ is residual.

(2) $\Rightarrow$ (3): Assume that there exists $N\in \N$ such that $Y_0=\{y\in Y: |\pi^{-1}(y)|=N\}$ is
residual. Note that $Y'=\overline{\{\pi^{-1}(y_0): y_0\in Y_c\}}$ is minimal and for all $y_0\in Y_c$,
$\pi^{-1}(y_0)$ is a minimal point of $(2^X, T)$. Thus we have (3) by taking $y\in Y_0\cap Y_c$.

(3) $\Rightarrow$ (1):  It is obvious.
\end{proof}


\begin{rem}
\begin{enumerate}
  \item Almost finite to one extensions can be defined not only for metric systems but also
  compact Hausdorff systems. We refer \cite{Shoenfeld} for more details, where it was
  called {\em generalized almost finite extension}.
  \item By definition it is obvious that a finite to one extension is almost finite to one.
  But in general, an almost finite to one extension may not be finite to one. For example,
  for Rees' example \cite{R}, $\pi: (X,T)\rightarrow (X_{eq},T)$ is an almost one to one
  extension, and for any $y\in X_{eq}$, either $|\pi^{-1}(y)|=1$ or $|\pi^{-1}(y)|=\infty$.
      \item  There is some example such that $\pi: X \rightarrow Y$ be a finite to one
      extension i.e. $y\in Y$, $|\pi^{-1}(y)|<\infty$, but $\sup_{y\in Y} |\pi^{-1}(y)|=\infty$ (see \cite[Example 5.7.]{YZ08}).
\end{enumerate}
\end{rem}

\begin{lem}\label{lem-f}
Let $A=\{x_1,x_2,\ldots,x_N\}\in 2^X$ and $\{t_i\}_i$ be a net of $T$. Then in $2^X$, we have that
$$\lim_i t_i A=\lim_i t_i \{x_1,\ldots,x_N \}=\{\lim_{i}t_ix_1,\ldots, \lim_{i}t_i x_N\}$$ if all limits exist.
\end{lem}

\begin{proof}
First by definition it is easy to see that $\{\lim_{i}t_ix_1,\ldots, \lim_{i}t_i x_N\}\subseteq \lim_i t_i A$.
Now let $x\in \lim_i t_i A$. By the definition of Hausdoff metric, for each $i$, there is some $z_i\in A$
such that $t_iz_i\to x$. Since $A$ is a finite set, one may assume $z_i$ is constant, i.e. $z_i=z\in A$.
Thus $x=\lim_i t_iz$. Thus $ \lim_i t_i A\subseteq\{\lim_{i}t_ix_1,\ldots, \lim_{i}t_i x_N\}$.
\end{proof}

By Lemma \ref{lem-f}, we have:
\begin{cor}\label{cor-f}
Let $\pi:(X,T)\rightarrow (Y,T)$ be an extension with $(Y,T)$ being minimal. If $\pi$ is an almost $N$
to one extension for some $N\in \N$, then the cardinality of each element of $Y'$ is $N$, where $(Y',T)$
is the minimal system defined at the beginning of this subsection.
\end{cor}

The following result is well known, and for completeness we include a proof.
\begin{lem} \label{finite-to-one}
Let $\pi: X \rightarrow Y$ be a finite to one extension (i.e. $\pi^{-1}(y)$ is finite for all $y\in Y$)
of the minimal systems $(X,T)$ and $(Y,T)$. Then the following conditions are equivalent:
\begin{enumerate}
  \item $\pi$ is open;
  \item $\pi$ is distal;
  \item $\pi$ is equicontinuous;
  \item $\pi$ is a factor of a finite group extension, i.e. there are extensions $\pi':Z\rightarrow X,
  \phi: Z\rightarrow Y$ such that $\phi=\pi\circ \pi'$ and $\phi$ is a finite group extension
  $$
  \xymatrix{
  X \ar[d]_{\pi} & Z \ar[l]_{\pi'} \ar[dl]^{\phi}      \\
  Y }
  $$
\end{enumerate}
In this case there exists $N\in \N$ such that $\pi$ is an $N$-to-1 map.
\end{lem}

\begin{proof}

(1)$\Rightarrow$ (3):
Fix $y\in Y$ such that $|\pi^{-1}(\{y\})|=\min_{z\in Y} |\pi^{-1}(\{z\})|$. Let $\pi^{-1}(\{y\})=\{x_1,\ldots,x_N\}$. From
openness of $\pi$ one has that $$\pi^{-1}(\{\lim_{i}{t_i}y\})=\lim_{i}t_i
\pi^{-1}(\{y\})=\{\lim_{i}t_ix_1,\ldots, \lim_{i}t_i x_N\}$$ for any sequence $\{t_i\}\subset T$ having that all limits exist, and that
the cardinality of the set is $N$.  Hence, by minimality, all
fibers of $\pi$ have the same cardinality $N$, which proves the
map is constant to one.

Now we show that $\pi$ is equicontinuous. First we have the following claim:

\medskip

{\noindent Claim : There is some $\ep_0>0$ such that if $(x,x')\in R_\pi$ and $d(x,x)<\ep_0$ then $x=x'$.}

\medskip

If Claim does not hold, then for any $k\in \N$ there is
$(x_k,x'_k)\in R_\pi$ with $d(x_k,x'_k)<1/k$ and $x_k\neq x_k'$. Let $y_k=\pi(x_k)=\pi(x'_k)$
and assume $y_k\to y \in Y, k\to\infty$. Since $\pi$ is open,
$\pi^{-1}(\{y_k\}) \to \pi^{-1}(\{y\}), k\to\infty$ in the Hausdorff topology. Note that $\pi$
is $N$ to 1 extension, and $|\pi^{-1}(y)|=|\pi^{-1}(y_k)|=N$ for all $k\in \N$. Let $\pi^{-1}(y)=\{x_1,\ldots,x_N\}$
and let $U_1,\ldots,U_N$ be disjoint closed neighbourhoods of $x_1,\ldots,x_N$ with $\d=\min_{1\le i\neq j\le N}d(U_i,U_j)>0$.
Then for large enough $k\in \N$ with $1/k<\d$, each $U_1,\ldots,U_N$ contains just one point of $\pi^{-1}(y_k)$
and $x_k,x_k'$ are in the different neighbourhoods which implies $d(x_k,x_k')>\d>1/k$. A contradiction!
Hence we have Claim.

If $\pi$ is not equicontinuous, then there exists $\ep>0$ such that for any $k\in \N$ there are
$(x_k,x'_k)\in R_\pi$ and $t_k\in T$ with $d(x_k,x'_k)<1/k$ and
$d({t_k}x_k, {t_k}x'_k)\ge \ep$. By Claim, this is impossible. Thus $\pi$ is
equicontinuous.

(3)$\Rightarrow$(2): It is obvious by definition.

(2)$\Rightarrow$(1): Given any $y_1,y_2\in Y$. Since $Y$ is minimal, there is a sequence $\{t_i\}_{i=1}^\infty$ such that
$\lim_{i\rightarrow \infty}t_iy_1=y_2$ and the limit $\lim_{i\rightarrow \infty} \pi^{-1}(t_iy_1)$ exists. Since $\pi^{-1}(t_iy_1)=t_i\pi^{-1}(y_1)$, $\pi$ is distal and $\pi^{-1}$ is a u.s.c. map, one has
$$|\pi^{-1}(y_2)|\ge |\lim_{i\rightarrow \infty} \pi^{-1}(t_iy_1)|=|\pi^{-1}(y_1)|.$$
By symmetry, one also has $|\pi^{-1}(y_1)|\ge |\pi^{-1}(y_2)|$. Thus
$$|\pi^{-1}(y_2)|= |\lim_{i\rightarrow \infty} \pi^{-1}(t_iy_1)|=|\pi^{-1}(y_1)|<\infty$$
and so $\pi^{-1}(y_2)=\lim_{i\rightarrow \infty} t_i\pi^{-1}(y_1)$.
This implies that
$Y'=\{\pi^{-1}(y):y\in Y\}$ and all
fibers of $\pi$ have the same cardinality since $y_1,y_2$ are arbitrary.
Since $Y'$ is closed, we have $$Y'=\{\pi^{-1}(y):y\in Y\}=\overline{\{\pi^{-1}(y): y\in Y\}}=\widetilde{Y},$$
 and so $Y_c=Y$, i.e. $\pi$ is open.

\medskip

Now have showed that (1)-(3) are equivalent. Next we show they are equivalent to (4).
It is easy to see that (4) implies (1). We will show that (2) implies (4). Let $\pi: X\rightarrow Y$ be an $N$ to 1 distal extension.
Fix a point $y_0\in Y$, and let $\pi^{-1}(y_0)=\{x_1,x_2,\ldots,x_N\}$.
Let $z_0=(x_1,x_2,\ldots, x_N)\in X^N$, and
$$Z=\overline{\O}(z_0,T)\subseteq X^N.$$
Since $\pi$ is distal, $z_0$ is a minimal point of $(X^N,T)$ and $(Z,T)$ is a minimal system.
Define $\pi': Z\rightarrow X$ by $(z_1,\ldots,z_N)\mapsto z_1$, i.e. the projection on the
first coordinate, and $\phi=\pi\circ \pi': Z\rightarrow Y$. It is easy to verify that $
\pi',\phi$ are factor maps. It is left to show $\phi$ is a finite group extension, that is there exists a
finite group $K$ such that $K$ acts continuously on $X$ from the right, and  the fibers of $\phi$ are the $K$-orbits in $Z$:
$\phi^{-1}(\{\phi(z)\})=zK$ for any $z\in Z$.

Let $E_{y_0}=\{\sigma: \{x_1,\ldots,x_N\}\rightarrow \{x_1,\ldots,x_N\}: \exists
\ t_i\in T \ \text{such that} \ t_iy_0\to y_0, t_ix_j\to \sigma(x_j), j=1,\ldots,N\}$.
It is easy to verify that $E_{y_0}$ is subgroup of the
permutation group on $\{x_1,\ldots,x_N\}$ and
\begin{equation}\label{s1}
  E_{y_0}x_0=\{\sigma(x_0): \sigma\in E_{y_0}\}=\{x_1,\ldots,x_N\}
\end{equation}
for all $x_0\in \{x_1,\ldots,x_N\}$.

For each $\sigma\in E_{y_0}$, define a map $H_\sigma: Z\rightarrow Z$ such that  for any convergent net $t_i\in T$ (i.e., the limit $\lim_i t_iz_0$ exist)
$$H_\sigma(\lim_i t_iz_0)= \lim_i t_i\sigma z_0,$$
where $\sigma z_0=(\sigma(x_1),\ldots, \sigma(x_N))$. First we need to show that $H_\sigma$ is
well defined. If there are nets $\{t_i\}_i$ and $\{t_i'\}_i$ such that $\lim_it_i z_0=\lim_it_i'z_0$.
Then for all $j\in \{1,\ldots,N\}$, $\lim_it_ix_j=\lim_it_i'x_j$. It follows that $\lim_it_i\sigma (x_j)=\lim_it_i' \sigma(x_j)$
for all $j\in \{1,\ldots,N\}$. That is, $$H_\sigma(\lim_i t_iz_0)= \lim_i t_i\sigma z_0= \lim_i t_i'\sigma z_0 =H_\sigma(\lim_i t_i'z_0),$$
and $H_\sigma$ is well defined.
In a similar way, we can verify that $H_\sigma$ is a homeomorphism for all $\sigma\in E_{y_0}$,
and $H_{\sigma\sigma'}=H_{\sigma}H_{\sigma'}$ for all $\sigma,\sigma'\in E_{y_0}$. Now let
$$K=\{H_\sigma: \sigma\in E_{y_0}\}.$$
Then $K$ is a finite group acts continuously on $X$ and by (\ref{s1}) we can show that
$\phi^{-1}(\{\phi(z)\})=zK$ for any $z\in Z$. That is, $\phi$ is finite group extension of $Y$.
\end{proof}

\begin{rem}

\begin{enumerate}
  \item For the equivalence of (1)-(3), see \cite{MS} or \cite[Chapter V, 6.5]{Vr}.
  In general, $\pi$ is equicontinuous if and only if it is a factor of a group
  extension \cite[Chapter 14, Theorem 1]{Au88}. Here (4) is only a special case
  of this result, and our proof of (4) follows from the one of  \cite[Chapter 14, Theorem 1]{Au88}.
  \item In general, an open finite to one extension may not be a finite group extension.
  For example, Let $G$ be a non-abelian finite group with a non normal subgroup $H$.
  Let $X=G/H$, $T=G$, and let $Y$ be the trivial system. Then $\pi: X\rightarrow Y$ is an open finite to one extension
      but not a finite group extension.
\end{enumerate}

\end{rem}

To end this section we cite Sacker-Sell's result, which give more information about finite to one extensions.

\begin{thm}\cite{SS} \label{Sacker-Sell}
Let $\pi: X \rightarrow Y$ be an extension of t.d.s. $(X,T)$ and $(Y,T)$.
Assume that $(Y,T)$ is minimal. Then the following conditions are equivalent:
\begin{enumerate}
  \item $\pi$ is distal and for some $y_0\in Y$, $|\pi^{-1}(y_0)|=N$, where $N\in \N$;
  \item $|\pi^{-1}(y)|=N$ for all $y\in Y$, where $N\in \N$;
  \item $X$ is an $N$-fold covering of $Y$, i.e. $|\pi^{-1}(y)|=N$ for all $y\in Y$ and
  for each $y\in Y$ there is an open neighbourhood $V$ of $y$ such that $\pi^{-1}(V)$
  consists of $N$ disjoint open sets $U_i$ and $\pi|_{U_i}: U_i\rightarrow V$ is a homeomorphism, $i=1,2,\ldots, N$.
\end{enumerate}
Finally, if any of these hold, then $X$ can be expressed as the disjoint union
$X=X_1\cup\ldots \cup X_k$ of compact minimal sets, where each $X_i$ an $n_i$-fold covering
of $Y$ and $n_1+\ldots +n_k=N$.
\end{thm}

In \cite{SS}, another main result is that an open finite-to-one extension of an
equicontinuous system is an equicontinuous one again as long as the phase
group was semicompactly generated (i.e., there is a compact $K\subset T$ such that every open $V\supset K$ generates $T$).

\section{Systems with finitely many ergodic measures}\label{sec-proof-main}

In this section we prove one of the the main results, i.e. Theorem A. According to Proposition \ref{almost N-1}
it remains to show

\begin{thm}\label{main-result-form1}
Let $(X,T)$ be a minimal system  with $T$ amenable group.
If $\pi: (X,T)\rightarrow (X_{eq},T)$ satisfies the following conditions:
\begin{enumerate}
  \item $\pi$ is almost $N$ to one for some $N\in \N$;
  \item $(X,T)$ is bounded non-tame, i.e. ${\rm IT}_k(X,T)\setminus \D^{(k)}=\emptyset$
  for some $k\ge 2$,
\end{enumerate}
then $|M^{erg}_T(X)|\le N (k-1)^N$. 
\end{thm}

\subsection{Some propositions and lemmas}\
\medskip

To prove Theorem \ref{main-result-form1}, we need the following result which was proved
to be true for $k=2$ in \cite [Proposition 3.3]{FGJO}.

\begin{prop}\label{independent sets}
Let $H$ be a locally compact second countable Hausdorff topological group
with left Haar measure $\Theta_H$, and let $k\in \N$ with $k\ge 2$. Suppose that
$V_{1},\ldots,V_{k}\subset H$ are compact subsets that satisfy
\begin{enumerate}
\item[(i)] $\overline{\text{int} \ V_i}=V_i$ for $i=1,2,\cdots,k$,

\item[(ii)] $\text{int}(V_{i})\cap \text{int}(V_{j})=\emptyset$ for all $1\le i\neq j\le k$,

\item[(iii)] $\Theta_{H}(\bigcap_{1\leq i\leq k}V_{i})>0$.
\end{enumerate}
Further, assume that $T\subset H$ is a dense subgroup and $\mathcal{G}\subset H$
is a residual set. Then there exists an infinite set $I\subset T$
such that for all $a\in\{1,2,\ldots,k\}^{I}$ there exists $h\in\mathcal{G}$
with the property that
\begin{equation}\label{eq: in the int}
h\in \bigcap_{t\in I} t^{-1} {\rm int}(V_{a_{t}}),\quad {\rm i.e.}\
th\in {\rm int}(V_{a_{t}}) \ \text{ for any }t\in I.
\end{equation}
\end{prop}

The proof of the above proposition will be given in the appendix.


\medskip

In this subsection, the group $T$ is assumed to be an  amenable group.
Let $(X,T)$ be a t.d.s. and $x_0\in X$. Let ${{\Phi}}=\{\Phi_N\}_{N\ge 1}$ be a F{\o}lner
sequence of $T$ and $\mu\in M_T(X)$. We say that $x_0$ is {\em generic for $\mu$ along $\Phi$} if
$$\frac{1}{|\Phi_N|}\sum_{t\in \Phi_N}\delta_{tx_0}\rightarrow \mu, \ \text{weakly$^*$ as }\ N\to\infty,$$
where $\d_x$ is the Dirac mass at $x$. This is equivalent to that for all $f\in C(X)$,
$$\frac{1}{|\Phi_N|}\sum_{t\in \Phi_N}f(tx_0)\rightarrow \int f \mu, \ N\to\infty.$$

A F{\o}lner sequence ${{\Phi}}=\{\Phi_N\}_{N\ge 1}$ is {\em tempered} if there exists a constant $C>0$ such that for all $N$
$$\left| \bigcup_{k<N}\Phi_k^{-1}\Phi_N\right|\le C|\Phi_N|.$$
Note that every F{\o}lner sequence  admits a tempered F{\o}lner subsequence \cite{L}.

\begin{thm}\cite{L}\label{Thm-L}
Let $T$ be an amenable group acting on a measure space $(X,\X,\mu)$ by measure preserving
transformation, and let $\Phi=\{\Phi_N\}_{N\ge 1}$ be a tempered F{\o}lner sequence.
Then for any $f\in L^1(X,\mu)$, there is a $T$-invariant $f^*\in L^1(X,\mu)$ such that
$$\lim_{N\to\infty}\frac{1}{|\Phi_N|}\sum_{t\in \Phi_N}f(tx)=f^*(x) \quad a.e.$$
In particular, if the $T$ action is ergodic,
$$\lim_{N\to\infty}\frac{1}{|\Phi_N|}\sum_{t\in \Phi_N}f(tx)=\int f(x)d\mu(x) \quad a.e.$$
\end{thm}

By Theorem \ref{Thm-L}, it is easy to show the following corollary.

\begin{cor}\label{Lem-H}
Let $(X,T)$ be a t.d.s. with $T$ amenable group, $\mu\in M_T^{erg}(X)$ and
$\Phi=\{\Phi_N\}_{N\ge 1}$ a tempered F{\o}lner sequence. Then $\mu$-almost every $x\in X$ is generic for $\mu$ along $\Phi$.
\end{cor}

We will not use the following lemma in the paper, and it is of independent interest.
\begin{lem}
Let $\pi: (X,T)\rightarrow (Y,T)$ be a group extension with respect to a group $K$,
where $T$ is amenable. If $\mu\in M^{erg}_T(X)$ is also $K$-invariant, then $\pi_*^{-1}(\pi_*\mu)=\{\mu\}$.

Thus if for each $\nu\in M^{erg}_T(Y)$ there is some ergodic element $\mu\in \pi_*^{-1}(\nu)$ is
$K$-invariant, then $|M_T^{erg}(X)|= |M_T^{erg}(Y)|$.
\end{lem}

\begin{proof}
The proof is similar to the proof of Proposition 3.10 of \cite{F}.
Recall that $\pi$ is a {\it group} extension if there exists a
compact Hausdorff topological group $K$ such that the following
conditions hold: $K$ acts continuously on $X$ from the right: the
right action $X\times K\rightarrow X$, $(x,k)\mapsto xk$ is
continuous and $t(xk)=(tx)k$ for any $t\in T$ and $k\in K$;
the fibers of $\pi$ are the $K-$orbits in $X$:
$\pi^{-1}(\{\pi(x)\})=xK$ for any $x\in X$.

Suppose that $x_0$ is any generic point for $\mu$ along some tempered F{\o}lner sequence $\Phi$.
Then for any $k\in K$ and any continuous $f\in C(X)$,
$$\lim_{N\to\infty}\frac{1}{|\Phi_N|}\sum_{t\in \Phi_N}f(tx_0k)=\int f(xk)d\mu(x) =\int f(x) d\mu(x) \quad a.e.$$
and so $x_0k$ is also a $\mu$-generic point. Thus, if $x_0$ is $\mu$-generic,
the whole fiber of $\pi(x_0)$ is $\mu$-generic.

Suppose now that $\eta$ is any ergodic measure on $X$ such that $\pi_*\eta=\pi_*\mu$.
Then by Corollary \ref{Lem-H}, there is some $\mu$-generic point $x_0\in X$ and
$\eta$-generic point $x_1\in X$ along the same tempered F{\o}lner sequence such that $\pi(x_0)=\pi(x_1)$.
By above the whole fiber of $\pi(x_0)$ is $\mu$-generic, and hence $x_1$ is also $\mu$-generic.
It follows that $\eta=\mu$. That is , $\pi_*^{-1}(\pi_*\mu)=\{\mu\}$. The proof is completed.
\end{proof}

What we will use is the following.

\begin{lem}\label{lem-finite-measures}
Let $\pi: (X,T)\rightarrow (Y,T)$ be a finite to one extension, where $T$ is amenable and there is
some $N$ such that $|\pi^{-1}(y)|\le N, \forall y\in Y$. If $(Y,T)$ has finitely many ergodic measures,
then so does $(X,T)$. In fact, $|M_T^{erg}(X)|\le |M_T^{erg}(Y)|\cdot N$.
\end{lem}

\begin{proof}

Let $\pi_*: M_{T}(X)\rightarrow M_T(Y), \mu\mapsto \mu\circ \pi^{-1}$. Then $\pi_*$ is a
surjective affine map and maps $M^{erg}_T(X)$ onto $M^{erg}_T(Y)$.
If $|M_T^{erg}(X)|> |M_T^{erg}(Y)|\cdot N$, then there are $\mu_1,\mu_2,\ldots , \mu_{N+1}\in M_T^{erg}(X)$  and $\nu\in M^{erg}_T(Y)$ such that
$$\pi_*(\mu_1)=\pi_*(\mu_1)=\ldots =\pi_*(\mu_{N+1})=\nu.$$
Suppose $W_i$ are the collection of $\mu_i$-generic points along some F{\o}lner sequence, $1\le i\le N+1$.
By Lusin theorem, $\pi(W_i)$ is an analysis set, therefore it is $\nu$-measurable. One has that
$$\displaystyle\nu \left(\bigcap_{1\leq i\leq N+1}\pi(W_i)\right)=1. $$
In particular, $\bigcap_{1\leq i\leq N+1}\pi(W_i)\neq \emptyset$ and let $y\in \bigcap_{1\leq i\leq N+1}\pi(W_i)$.

As $|\pi^{-1}(z)|\le N, \forall z\in Y$, it follows that there are $i\neq j\in \{1,2,\ldots,N+1\}$
such that $W_i\cap W_j\neq \emptyset$. It contradicts with the fact
$\mu_i \perp \mu_j$. Thus we have that $|M_T^{erg}(X)|\le |M_T^{erg}(Y)|\cdot N$.
\end{proof}

\subsection{Proof of Theorem \ref{main-result-form1}}\
\medskip

In the subsection we give the proof of Theorem \ref{main-result-form1}.
The following proposition is a key step to the proof of Theorem \ref{main-result-form1}.
\begin{prop}\label{key-prop-numberofIT} Let $(Y',T)$ be a minimal t.d.s. and $\tau: (Y',T)\rightarrow (X_{eq},T)$ be the factor
to the its maximal equicontinuous factor. If $\tau$ is almost one to one, and ${\rm IT}_l(Y')\setminus \D^{(l)}=\emptyset$
for some $l\ge 2$, then $|M^{erg}_T(Y')|\le l-1$.

\end{prop}
\begin{proof}
By Theorem \ref{Thm-Ellis}, let $Y=X_{eq}=G/\Gamma$, where $G=E(X_{eq})$ is the Ellis semigroup and
$\Gamma$ is a closed subgroup of $E(X_{eq})$. Since $(X_{eq},T)$ is equicontinuous, $G=E(X_{eq})$ is a compact Hausdorff group.
Let the left Haar probability measure of $G$ be $\Theta_G$, and the probability measure $m$ induced by $\Theta_G$ is
the unique $T$-invariant probability measure of $(Y,T)$. Let $\phi: G\rightarrow Y=G/\Gamma$ be the factor map.
Then $\phi$ is open and $m=\phi_*(\Theta_G)=\Theta_G\circ \phi^{-1}$.
$$
\xymatrix{
  Y' \ar[d]_{\tau}  & (G,\Theta_G)  \ar[dl]_{\phi}     \\
  (Y=G/\Gamma,m) }
$$

Set for $k\in \N$
$$X_k=\{x\in Y': |\tau^{-1}(\tau(x))|=k\}, \quad Y_k=\tau(X_k);$$
$$X_\infty=\{x\in Y': |\tau^{-1}(\tau(x))|=\infty\}, \quad Y_\infty=\tau(X_\infty).$$
By definition, $\{X_k\}_{k\in \N\cup\{\infty\}}$ is a disjoint family of $X_{eq}$ and $\{Y_k\}_{k\in \N\cup\{\infty\}}$
is a disjoint family of $Y$. We remark that $Y_1$ is a dense $G_\delta$-set as $\tau$ is almost one to one.

\medskip

Let $2^{Y'}$ be the hyperspace of $Y'$ with Haudorff metric $d_H$. Consider the map
$$F=\tau^{-1}: Y\rightarrow 2^{Y'}, \quad y\mapsto \tau^{-1}(\{y\}).$$
Then $F$ is upper semi-continuous and hence Borel measurable. And define
$$G: 2^{Y'}\rightarrow \N\cup\{0,\infty\},\quad  A\mapsto |A|.$$
Recall that the Hausdorff  metric on $2^{Y'}$ is defined as follows:
$$d_H(A,B)=\sup_{y\in A} \inf_{y'\in B}d_{Y'}(y,y')+\sup_{y'\in B}\inf_{y\in A} d_{Y'}(y,y').$$
Let $A_0\in 2^{Y'}$. If $|A_0|=M$ for some finite $M\in \mathbb{N}$, then we may put $A_0=\{y_1,\ldots,y_M\}$.
Let $\displaystyle\epsilon_0=\min_{1\leq i<j\leq M}d_{Y'}(y_i,y_j)$. Notice that if $d_H(A_0,A)<\frac{\epsilon_0}{2}$,
then $|A|\geq M$.
If $|A_0|=\infty$, then for arbitrary $M'$, one can choose $\{y'_1,\ldots,y'_{M'}\}\subset A_0$.
Similarly find $\displaystyle\epsilon_{M'}=\min_{1\leq i<j\leq M'}d_{Y'}(y'_i,y'_j)$, then
$|A|\geq M'$ if $d_H(A_0,A)<\frac{\epsilon_{M'}}{2}$.  Equivalently, this can be expressed as
$$\displaystyle \liminf_{A\to A_0}G(A)\geq G(A_0).$$
That is, $G$ is lower semi-continuous. It follows that $G\circ F: Y\rightarrow \N\cup\{0,\infty\}$
is Borel measurable. Note that $Y_k={(G\circ F)}^{-1}(k)$, $Y_k$ is Borel measurable for each
$k\in \N\cup\{0,\infty\}$. And hence $X_k=\tau^{-1}(Y_k)$ is Borel measurable for each
$k\in \N\cup\{0,\infty\}$. Thus $\{X_k\}_{k\in \N\cup\{\infty\}}$ and $\{Y_k\}_{k\in \N\cup\{\infty\}}$ are measurable and $T$-invariant.

\medskip

Since  $\{Y_k\}_{k\in \N\cup\{\infty\}}$ are disjoint measurable $T$-invariants and $m$ is
ergodic, there is only one $k_0 \in \N \cup \{ \infty \}$ such that $m(Y_{k_0})=1$ and $m(Y_k)=0$ for all $k\neq k_0$.

\medskip

Now we show that $(Y',T)$ has at most $l-1$ ergodic measures. If not, assume that $\mu_1,\mu_2,\ldots,\mu_l$
be $l$ distinct ergodic measures of $(Y',T)$. Let $W_i$ be the set of
$\mu_i$-generic points for $i\in \{1,2,\ldots, l\}$. By Lusin theorem, $\{\tau(W_i)\}_{i=1}^l$ are
universally measurable as they are analytic sets. Since $(Y,T)$ is uniquely ergodic,
one has that $$\tau_*(\mu_1)=\ldots=\tau_*(\mu_l)=m.$$ It follows that
$m(\tau(W_1))=m(\tau(W_2))=\ldots =m(\tau(W_l))=1$ and hence
$$m(\bigcap_{i=1}^l\tau(W_i)\cap Y_{k_0})=1.$$ In particular,
$\bigcap_{i=1}^l\tau(W_i)\cap Y_{k_0}\neq \emptyset$. Let $z\in \bigcap_{i=1}^l\tau(W_i)\cap Y_{k_0}$.
Thus $|\tau^{-1}(z)|\ge l$ and we have that $k_0\ge l$.

\medskip

Recall $2^{Y'}=\{A\subset Y': A \ \text{compact and non-empty}\}$ and $F: Y\rightarrow 2^{Y'}, y\mapsto \tau^{-1}(\{y\}).$
Then $F$ is upper semi-continuous and hence measurable. By Lusin's Theorem, there is some compact
set $K\subseteq Y_{k_0}$ such that $m(K)>0$ and $F|_{K}: K\rightarrow 2^{Y'}$ is continuous.
Let $m|_K$ be the measure restricted on $K$ of $m$. Since $m(Y_{k_0})=1$, one has that
$$K\cap Y_{k_0}\cap {\rm supp} (m|_K)\neq \emptyset.$$
Let $y_0\in K\cap Y_{k_0}\cap {\rm supp} (m|_K)$. Then $|\tau^{-1}(y_0)|=k_0\ge l$ and $m(V\cap K)>0$
for any neighbourhood $V$ of $y_0$.

\medskip

Choose distinct elements $\xi_1,\xi_2,\ldots,\xi_l\in \tau^{-1}(y_0)$ and let
$\ep=\frac 14 \min_{1\le i\neq j\le l}d(\xi_i,\xi_j)$. Let $U_i=\overline{B_\ep(\xi_i)}$
for $1\le i\le l$. Then $U_i$ for $1\le i\le l$ is proper, i.e,  $U_i$ is  a compact subset with $\overline{int(U_i)} = U_i$.
We will show that $U_1,U_2,\ldots,U_l$ is an infinite independent tuple of
$(Y,T)$, i.e. there is some infinite set $I\subseteq T$ such that
$$\bigcap_{t\in I}t^{-1}U_{a_t}\neq \emptyset, \  \text{for all} \  a\in \{1,2,\ldots,l\}^I.$$
Let $V_i'=\tau(U_i)$ for $1\le i\le l$. By Lemma \ref{lem:proper one to one}, $V_i'$ is proper
for each $i\in \{1,2,\ldots,l\}$, i.e. $\overline{{\rm int}(V_i') }=V_i'$.

We claim that ${\rm int }(V_i')\cap {\rm int}(V_j')=\emptyset$ for all $1\le i\neq i\le l$.
In fact, if there is some  $1\le i\neq j\le l$ such that
${\rm int }(V_i')\cap {\rm int}(V_j')\not=\emptyset$, then
$${\rm int }(V_i')\cap {\rm int}(V_j')\cap Y_1\not=\emptyset,$$
as $Y_1$ is a dense $G_\d$ set. Let $y'\in {\rm int }(V_i')\cap {\rm int}(V_j')\cap Y_1$.
Then there are $x_i\in U_i$ and $x_j\in U_j$ such that $y'=\tau(x_i)=\tau(x_j)$, which contradict with $y'\in Y_1$.

\medskip

Since $F$ is continuous on $K$, one can choose $\d>0$ such that for any $y\in B_\d(y_0)\cap K$
one has that $d_H(F(y),F(y_0))<\ep$. By the definition of Hausdorff metric, the fibre $F(y)=\tau^{-1}(y)$ intersects all $U_1,\ldots, U_l$, so that
$$y\in \bigcap_{i=1}^l V_i'.$$
Therefore, $B_\d(y_0)\cap K\subseteq \bigcap_{i=1}^l V_i'$ and hence
$$m(\bigcap_{i=1}^l V_i')\ge m(B_\d(y_0)\cap K)>0.$$
Set $V_i=\phi^{-1}(V_i')$ for all $1\le i\le l$. Since $\phi$ is open, $V_1,\ldots,V_l$ are proper,
$$\text{int}(V_i)\cap \text{int}(V_j)=\phi^{-1}(\text{int}(V_i'))\cap \phi^{-1}(\text{int}(V_j'))=\emptyset$$
for $1\le i\neq j\le l$
 and
$$\Theta_G(\bigcap_{i=1}^l V_i)=m(\bigcap_{i=1}^l V_i')>0.$$
Let $\G=\phi^{-1}(Y_1)$. Then $\G$ is also residual. By Proposition \ref{independent sets},
there is an infinite $I\subseteq T$ such that for all $a\in\{1,2,\ldots,k\}^{I}$ there exists $h\in\mathcal{G}$
with the property that
\begin{equation*}
h\in \bigcap_{t\in I} t^{-1} {\rm int}(V_{a_{t}}).
\end{equation*}
Hence
\begin{equation*}
\phi(h)\in \bigcap_{t\in I} t^{-1} {\rm int}(V'_{a_{t}}).
\end{equation*}
Note that $\phi(h)\in Y_1$, $|\tau^{-1}(\phi(h))|=1$. Set $\tau^{-1}(\phi(h))=\{x_0\}$. Then
\begin{equation*}
x_0\in \bigcap_{t\in I} t^{-1} U_{a_{t}}.
\end{equation*}
That is, $(\xi_1,\xi_2,\ldots,\xi_l)$ is a $l$-IT-tuple over $\tau$, a contradiction! The proof is completed.
\end{proof}

{\bf Now we are ready to show Theorem \ref{main-result-form1}.}

\begin{proof}[Proof of Theorem \ref{main-result-form1}]
Assume that $\pi: (X,T)\rightarrow (X_{eq},T)$ is almost $N$ to one, and $X$ is
bounded non-tame, i.e. ${\rm IT}_k\setminus \D^{(k)}=\emptyset$ for some $k\ge 2$. Let $Y=X_{eq}$.
We will divide the proof into the following steps.

\medskip
\noindent{\bf Step 1: Lift $\pi$ to an open $N$ to one map through almost one to one extensions.}

\medskip

Recall that $(2^{X},T)$ is the hyperspace system of $(X,T)$, which is defined by
$$T\times2^X \rightarrow 2^X:\ (t, A)\mapsto tA=\{ta: a\in A\}, \ t\in T,\ A\in 2^X.$$
Let $M_N(X)=\{A\in 2^X: |A|\le N\}$. Then it is easy to see that $M_N(X)$ is a $T$-invariant and
closed subset of $2^X$, and hence $(M_N(X),T)$ is a subsystem of $(2^X,T)$.

Let $Y_c$ be the set of continuous points of $\pi^{-1}: Y\rightarrow 2^X$.
By Proposition \ref{almost N-1}, for each point $y\in Y_c$ one has that $\pi^{-1}(y)$ is a minimal point of $(2^X, T)$ and $|\pi^{-1}(y)|=N$. Let
$$Y'=\overline{\{\pi^{-1}(y):y\in Y_c\}}.$$
By Subsection \ref{subsec-finite-one}, $(Y',T)$ is minimal and by Corollary \ref{cor-f} $|A|=N$
for all $A\in Y'$. Hence $(Y',T)$ is a minimal subsystem of $(M_N(X),T)$. Note that
for each $A\in Y'$, there is some $y\in Y$ such that $A\subseteq \pi^{-1}(y)$, and hence $A\mapsto y$ define
$\tau: Y'\rightarrow Y$ such that for all $y\in Y_c$, $\tau(\pi^{-1}(y))=y$. Since $\pi^{-1}$
is continuous at points of $Y_c$, $\tau^{-1}(y)=\{\pi^{-1}(y)\}$ for all $y\in Y_c$, i.e. $\tau$ is almost one to one.

If $A=\{x_1,\ldots,x_N\}\in Y'$, it is easy to verify that
$$\tau:Y'\rightarrow Y: A=\{x_1,\ldots,x_N\}\rightarrow \pi(x_1).$$

Let
$$X'=\{(x,A)\in X\times Y': x\in A\}.$$
Then $(X',T)$ is a subsystem of $(X\times Y', T)$. Let $\sigma$ and $\pi'$ be the projections:
$$\sigma : X'\rightarrow X, (x,A)\mapsto x; \ \text{and}\ \pi': X'\rightarrow Y', (x,A)\mapsto A. $$
For each $y\in Y_c$ and $x\in \pi^{-1}(y)$, $\sigma^{-1}(x)=\{(x, \pi^{-1}(y))\}$, i.e. $\sigma$
is an almost one-to-one extension. Notice that for $A=\{x_1,\ldots,x_N\}\in Y'$,
$${\pi'}^{-1}(A)=\{(x,A): x\in A\}=\{(x_i, \{x_1,\ldots,x_N\}): 1\le i\le N\}.$$
It follows that $\pi':(X',T)\rightarrow (Y',T)$ is an open $N$ to one extension.

To sum up, we have the following diagram:
\begin{equation*}
  \xymatrix{
  X \ar[d]_{\pi}
                & X' \ar[d]^{\pi'} \ar[l]_{\sigma} \\
  Y
                & Y'   \ar[l]_{\tau}          }
\end{equation*}
where $\sigma$ and $\tau$ are almost one-to-one extensions, $\pi'$ is an open $N$ to one extension.
\footnote{See \cite{V70} for the construction of the general case.}

\medskip
\noindent {\bf Step 2: The length of any IT-tuple for $Y'$ is not more than $(k-1)^N$.}

\medskip

Let  $(X^N,T)$ be the product system. Define
$$p:X^N\rightarrow M_N(X): (x_1,\ldots,x_N)\mapsto \{x_1,\ldots,x_N\},$$
which is a continuous map.
We also have the following commuting diagram

\begin{equation*}
  \xymatrix{
  X^N \ar[d]_{p} \ar[r]^{T}
                 & X^N \ar[d]^{p} \\
  M_N(X) \ar[r]^{T}
                   & M_N(X)  }
\end{equation*}

Since there is some integer $k\ge 2 $ such that $(X,T)$ has no $k$-IT-tuples, we claim that 
there is some $l\leq (k-1)^N$ such that ${\rm IT}_{l+1}(Y',T)\setminus \triangle^{(l+1)}(Y')=\emptyset$,
i.e. $(Y',T)$ has no $l+1$ IT-tuples. 

Indeed, suppose that there is
$l>(k-1)^N$ and $(y'_1,\ldots,y'_l)\in {\rm IT}_l(Y',T)\setminus \triangle^{(l)}(Y')$. Then 
there exists $y\in Y$ such that $\tau(y'_j)=y$ for $1\leq j\leq l$.
From Proposition \ref{factor and IT tuple} (4), there exists an $l$-IT-tuple $(c_1,\ldots,c_l)\in (X^N)^l$
such that $p(c_j)=y'_j$ for $1\leq j\leq l$. Suppose $c_j=(x_i^{(j)})_{1\leq i\leq N}$.
Notice that $\pi(x_i^{(j)})=\tau(y'_j)=y$ for any $1\le i\le N,1\le j\le l$.

As $l>(k-1)^N$, there exists some $1\le i\le N$ such that $|\{x_i^{(j)}:1\leq j\leq l\}|\geq k$.
Thus one can choose a $k$-tuple $(c_{n_1},\ldots,c_{n_{k}})$ such that $x_i^{(n_r)}\neq x_i^{(n_s)}$ for
$1\leq r<s\leq k$. For each $1\le i\le N$, let $p_i:X^N\rightarrow X$ be the projection to the $i$-th coordinate.
It is clear that $p_i$ is a factor map, which implies that $(x_i^{(n_1)},\ldots, x_i^{(n_{k})})$ is a $k$ IT-tuple by
Proposition \ref{factor and IT tuple} (4) again. This is a contradiction, and thus the claim is proved.



\medskip
\noindent {\bf Step 3: Count the member of ergodic measures of $X$.}

\medskip

Thus if one can show that $Y'$ has finitely many ergodic measures, then by Lemma~ \ref{lem-finite-measures}
so does $X$. 
 In fact we will show that $|M^{erg}_T(Y')|\le l\le (k-1)^N$.
Hence by Lemma~ \ref{lem-finite-measures},
$$|M^{erg}_T(X)|\le |M^{erg}_T(X')| \le N|M^{erg}_T(Y')| \le Nl\le N(k-1)^N.$$

Now it is left to show that $|M^{erg}_T(Y')|\le l\le (k-1)^N$. Assume the contrary that $|M^{erg}_T(Y')|\ge (k-1)^N+1$.
Then by Proposition \ref{key-prop-numberofIT}, ${\rm IT}_{(k-1)^N+1}\setminus \D^{((k-1)^N+1)}\not=\emptyset$.
This contracts to the claim in {\bf Step 2}. So, $|M^{erg}_T(Y')|\le l\le (k-1)^N$, and the proof is completed.
\end{proof}


\subsection{A Corollary}\
\medskip

Corollary B
follows from the following result:
\begin{cor}\label{cor-finite-to-one-ex}
Let $(X,T)$ and $(Y,T)$ be  minimal systems with $T$ amenable. If $\pi: (X,T)\rightarrow (Y,T)$ is finite to one
and $(Y,T)$ has finitely many ergodic measures, then so does $(X,T)$.
\end{cor}

\begin{rem}
Note that if $\sup_{y\in Y}|\pi^{-1}(y)|<\infty$, then the result is obvious by
Lemma \ref{lem-finite-measures}. But there is some system, for all $y\in Y$,
$|\pi^{-1}(y)|<\infty$, but $\sup |\pi^{-1}(y)|=\infty$ (see \cite[Example 5.7.]{YZ08}). We need to take care of this case.
\end{rem}

\begin{proof}[Proof of Corollary \ref{cor-finite-to-one-ex}]
As in the proof of Theorem \ref{main-result-form1}, set for $k\in \N$
$$Y_k=\{y\in Y: |\pi^{-1}(y)|=k\};$$
$$Y_\infty=\{y\in Y: |\pi^{-1}(y)|=\infty\}.$$
Then $\{Y_k\}_{k\in \N\cup\{\infty\}}$ are disjoint measurable $T$-invariants sets.
Since $\pi$ is finite to one, $Y_\infty=\emptyset$ and $Y=\cup_{k\in \N} Y_k$.
For each ergodic measure $m$ on $Y$, there is only one $k_0 \in \N $ such that
$m(Y_{k_0})=1$ and $m(Y_k)=0$ for all $k\neq k_0$. Similar to the proof of
Lemma \ref{lem-finite-measures}, there are finitely many ergodic measures
in $\pi_*^{-1}(m)$. Since $(Y,T)$ has finitely many ergodic measures,
it follows that $(X,T)$ has only finitely many ergodic measures.
\end{proof}





\section{The structure of bounded systems}\label{sec-SequenceEntropy}



In this section we will prove Theorem C. 
First we will give the structure of bounded systems, that is:

\begin{thm}\label{Main-bd}
If $(X,T)$ is a bounded minimal system with $T$ abelian, then it is an almost finite to one extension of its maximal equicontinuous factor.
\end{thm}

\begin{rem}\label{N'N}
One may give a more precise version of Theorem \ref{Main-bd} as follows:
\medskip

{\em If $(X,T)$ is a minimal system with $T$ abelian and $h_\infty(X,T)=\log N$, then it is an almost $N'$ to one extension of its maximal equicontinuous factor, where $N'\le N$.}

\medskip

The fact $N'\le N$ follows from Theorem \ref{MS-lemma}. In fact, if $\pi: X\rightarrow X_{eq}$ is almost $N'$ to one, then by Proposition \ref{almost N-1}, there exists $y_0\in Y$ such that $|\pi^{-1}(y_0)|=N'$ and $\pi^{-1}(y_0)$ is a minimal point of $(2^X, T)$. Let $\pi^{-1}(y_0)=\{x_1,\ldots,x_{N'}\}$. Then by Theorem \ref{MS-lemma}, $(x_1,\ldots,x_{N'})\in SE_{N'}(X,T)$. Thus $\log N'\le h_\infty(X,T)=\log N$ by Lemma \ref{exist of se}, and hence $N'\le N$.

For the Morse minimal system $(X,T)$, $h_\infty(X,T)=\log 2$, and it is almost 2 to one extension of its maximal equicontinuous factor.
In general, one can not get that $N'=N$. For example, for the substitution minimal system $(X,T)$ in \cite[Proposition 5.1]{Go},
$h_\infty(X,T)=\log 2$, but it is an almost one to one extension of its maximal equicontinuous factor.
\end{rem}

We also remark that Theorem \ref{Main-bd} is in fact an improvement of a previous result
obtained by Maass and Shao in \cite{MS}.

\begin{prop}\cite{MS} \label{MS-main}
Let $(X,T)$ be a minimal system with $T$ abelian. If $(X,T)$ is bounded, then $(X,T)$ has the following
structure:
\[
\begin{CD}
X @<{\sigma'}<< X'\\
@VV{\pi}V      @VV{\pi'}V\\
Y=X_{eq} @<{\tau' }<< Y'
\end{CD}
\]
where $X_{eq}$ is the maximal equicontinuous factor of $X$, $\sigma'$ and $\tau'$ are proximal extensions, and $\pi'$ is a
finite to one equicontinuous extension.
\end{prop}

\begin{rem}
In \cite{MS}, all results are stated under $\Z$-actions, and they hold for systems with $T$ abelian.
\end{rem}

Thus, to prove Theorem \ref{Main-bd}, we need to show $\sigma'$ and $\tau'$ in Proposition \ref{MS-main} are
actually almost one-to-one, and hence $\pi$ is almost finite to one. To do this, first we need some notions introduced in \cite{HKZ,Z}.

\medskip
Let $(X,T)$ be a t.d.s.. and let $U\subseteq X$ be a non-empty open subset, $\d>0$ and $r\in \N$ with $r\ge 2$. Set
\begin{equation}\label{}
  N(U,\d; r)=\{t\in T: \exists x_1, x_2, \ldots, x_r\in U \ \text{such that }\ \min_{1\le i\neq j\le r} d(tx_i, tx_j)>\d\}.
\end{equation}

\begin{de}
Let $(X,T)$ be a t.d.s.. and $r\in\N$ ($r\ge 2$). We say that $(X,T)$ is {\em multiple $r$-sensitive} if there is some $\d>0$ such that
for any $k\in\N$ and any finite non-empty open subsets $U_1, U_2,\ldots, U_k$ of $X$, $\bigcap_{i=1}^k N(U_i,\d;r)\neq \emptyset$.
\end{de}

The following proposition relates the multiple sensitivity and sequence entropy. 

\begin{prop}
Let $(X,T)$ be a t.d.s. with $T$ abelian, and let $r\in \N$ with $r\ge 2$. If $(X,T)$ is multiple $r$-sensitive,
then there is some sequence $A\subset T$ such that $h_A(X,T)\ge \log r$.


In particular, $h_\infty(X,T)\ge \log r$.

\end{prop}

\begin{proof}
The proof follows the arguments of the proof of in \cite[Theorem 1.4]{Z} for $\Z$-actions.
Since $(X,T)$ is multiple $r$-sensitive, there is some $\d>0$ such that $\bigcap_{i=1}^k N(U_i,2\d;r)\neq \emptyset$
for any finite non-empty open subsets $U_1, U_2,\ldots, U_k$ of $X$.
Let $\a$ be an open cover of $X$ with ${\rm diam} (\a)<\d$. We will show that there is some sequence $A\subset T$ such that $h_A(T,\a)\ge \log r$.

\medskip
\noindent {\bf Claim:} {\em There exist a sequence $\{t_n\}_{n\in \N}$ of $T$ and a set of non-empty open subsets
$\{V_s\}_{s\in \Omega}$ of $X$, where $\Omega=\bigcup_{i=1}^\infty \{1,2,\ldots, r\}^i$ such that
\begin{enumerate}
  \item if $t$ is a sub-word of $s$, then $V_s\subseteq V_t$;
  \item for each $m\in \N$ one has that
  \begin{equation*}
  \min_{s\neq s'\in \{1,2,\ldots,r\}^m }\max _{1\le i\le m} {\rm dist} \big( {t_i} V_s, {t_i} V_{s'} \big)>\d.
\end{equation*}
\end{enumerate}

}


\begin{proof}[Proof of Claim]
We prove the claim by induction. Since $(X,T)$ is multiple $r$-sensitive, there is some $t_1\in T$ and $x_1,x_2,\ldots, x_r\in X$ such that
$$\min_{1\le s_1\neq s_1'\le r } d({t_1} x_{s_1}, {t_1}x_{s_1'})>2\d.$$
Then we choose open neighborhood $U_{s_1}$ of ${t_1}x_{s_1}$ for all $s_1\in \{1,2,\ldots,r\}$ such that
\begin{equation*}
\min_{1\le s_1\neq s_1'\le r } {\rm dist} (U_{s_1}, U_{s_1'})>\d.
\end{equation*}
Set $V_{s_1}=t_{1}^{-1}U_{s_1}$ for $s_1\in \{1,2,\ldots,r\}$. Then we have that
\begin{equation*}
  \min_{s\neq s'\in \{1,2,\ldots,r\} }\max_{i=1} {\rm dist}\big( {t_i}(V_s), {t_i}(V_{s'})\big)>\d.
\end{equation*}
Thus we have our first step.

Now assume that there exist a sequence $t_1, t_2, \ldots, t_m$ of $T$ and a set of
non-empty open subsets $\{V_s\}_{s\in \Omega_m}$ of $X$, where $\Omega_m=\bigcup_{i=1}^m \{1,2,\ldots, r\}^i$ such that
\begin{enumerate}
    \item[$(1)_m$] if $t\in \Omega_m$ is a sub-word of $s\in \Omega_m$, then $V_s\subseteq V_t$;
    \item[$(2)_m$] for each $j\in \{1,2, \ldots,m\}$ one has that
\begin{equation*}
  \min_{s\neq s'\in \{1,2,\ldots,r\}^j }\max _{1\le i\le j} {\rm dist} \big( {t_i}V_s , {t_i}V_{s'}\big)>\d.
\end{equation*}
  \end{enumerate}
Since $(X,T)$ is multiple $r$-sensitive,
\begin{equation*}
  \bigcap_{s\in \{1,2,\ldots, r\}^m} N(V_s,2\d;r)\neq \emptyset.
\end{equation*}
Pick $t_{m+1}\in\displaystyle   \bigcap_{s\in \{1,2,\ldots, r\}^m} N(V_s,2\d;r)$. By the definition,
for each $s\in \{1,2,\ldots, r\}^m$, one can find $x_{(s,1)},x_{(s,2)},\ldots, x_{(s,r)}\in V_s$ such that
\begin{equation*}
  \min_{1\le s_{m+1}\neq s'_{m+1}\le r} d({t_{m+1}}x_{(s,s_{m+1})}, {t_{m+1}}x_{(s,s'_{m+1})})>2\d.
\end{equation*}
Then we choose a non-empty open neighbourhood $U_{(s,s_{m+1})}$ of ${t_{m+1}}x_{(s,s_{m+1})}$ for each
$(s,s_{m+1})\in \{1,2,\ldots,r\}^{m+1}$ with
\begin{equation*}
  \min_{1\le s_{m+1}\neq s'_{m+1}\le r} {\rm dist} (U_{(s,s_{m+1})}, U_{(s,s'_{m+1})})>\d.
\end{equation*}
Set $V_{(s,s_{m+1})}=V_s\cap t_{m+1}^{-1}U_{(s,s_{m+1})}$ for all $(s,s_{m+1})\in \{1,2,\ldots,r\}^{m+1}$.
Then it is easy to verify that
\begin{equation*}
  \min_{s\neq s'\in \{1,2,\ldots,r\}^{m+1} }\max _{1\le i\le m+1} {\rm dist} \big( {t_i}V_s, {t_i}V_{s'}\big)>\d.
\end{equation*}
The proof of Claim is completed.
\end{proof}

For any $m\in \N$ and any $ s, s'\in \{1,2,\ldots,r\}^m$, one has that
\begin{equation*}
  \begin{split}
  & \min_{s\neq s'\in \{1,2,\ldots,r\}^m }\max _{1\le i\le m} d \big( {t_i}(x_s), {t_i}(x_{s'})\big)\\ \ge & \min_{s\neq s'\in \{1,2,\ldots,r\}^m }
  \max _{1\le i\le m} {\rm dist} \big( {t_i}V_s, {t_i}V_{s'}\big)\\ >& \d .
  \end{split}
\end{equation*}
Since ${\rm diam} \a<\d$, $x_s$ and $x_{s'}$ will not be in the same element of $\bigvee_{i=1}^m t_i^{-1}\a$ whenever $s\neq s'\in \{1,2,\ldots,r\}^m$.
Thus, $N(\bigvee_{i=1}^m t_i^{-1}\a)\ge r^m$, which implies that $h_A(T,\a)\ge \log r$, where $A=\{t_i\}_{i=1}^\infty$. In particular, $h_\infty(X,T)\ge h_A(X,T)\ge \log r$.
\end{proof}

We also need the following lemmas to show the next proposition.

\begin{lem}\cite[Remark 8.]{C}\label{clay}
Let $\pi_i: (X_i,T)\rightarrow (Y_i,T)$ be proximal extensions for $i\in I$. Then $\prod_{i\in I}\pi_i: \prod_{i\in I}X_i\rightarrow \prod_{i\in I}Y_i$ is also proximal.
\end{lem}

\begin{lem}\cite[Lemma 6.17]{Au88}\label{Auslander}
Let $(X,T)$ be a topological system, let $x\in X$ and let $y$ be a minimal point, with $x$ and $y$ proximal. Let $U$ be a neighbourhood of $y$.
Then there is a $t\in T$ such that $tx,ty\in U$.
\end{lem}

\begin{lem}\label{lem-dist}
Let $\pi: (X,T)\rightarrow (Y,T)$ be an extension between t.d.s. with $(Y,T)$ being minimal. If $\pi$ is not almost finite to one, then for each fixed $r\ge 2$, there exist a constant $\d_r>0$ such that for each $y\in Y$,
there are $x_1,\ldots, x_r\in \pi^{-1}(y)$ with $$\min_{1\le i\neq j\le r} d(x_i, x_j)>\d_r.$$
\end{lem}

\begin{proof}
Fix $r\ge 2$ and let $\pi^{-1}: Y\rightarrow 2^X, y\mapsto \pi^{-1}(y)$.
Then $\pi^{-1}$ is a u.s.c. map, and by Theorem \ref{Kura}, the set $Y_c$ of
continuous points of $\pi^{-1}$ is a dense $G_\d$ subset of $Y$.
Let $y_0\in Y_c$ be a continuous point of $\pi^{-1}$.

Define $f: X^r\rightarrow \R$ as follows
$$(w_1,\ldots,w_r)\mapsto \min_{1\le i\neq j\le r} d(w_i, w_j).$$
It is easy to verify that $f$ is a continuous function.  Since $\pi$ is not almost finite to one and $(Y,T)$ is minimal, $\pi^{-1}(y_0)$ is an infinite set. Choose distinct points $z_1,z_2,\ldots, z_r\in \pi^{-1}(y_0)$. Then
$$f(z_1,\ldots,z_r)=\min_{1\le i\neq j\le r} d(z_i, z_j)>0.$$
Let $\d=\frac{1}{2} f(z_1,\ldots,z_r)$. By the continuity of $f$, there are open neighborhoods $U_1,\ldots,U_r$ of $z_1,\ldots,z_r$ such that for all $(z_1',\ldots,z_r')\in U_1\times \ldots \times U_r$,
$$|f(z'_1,\ldots,z'_r)-f(z_1,\ldots,z_r)|<\d.$$
In particular, we have that for all $(z_1',\ldots,z_r')\in U_1\times \ldots \times U_r$,
\begin{equation}\label{s2}
  \min_{1\le i\neq j\le r} d(z'_i, z'_j)=f(z'_1,\ldots,z'_r)>f(z_1,\ldots,z_r)-\d=\d.
\end{equation}

Since $y_0\in Y_c$ is a continuous point of $\pi^{-1}$, there is an open neighbourhood $V$ of $y_0$ such that for all $y'\in V$,
\begin{equation}\label{s-s2}
\pi^{-1}(y')\cap U_j\neq \emptyset, \ \forall 1\le j\le r.
\end{equation}

Since $(Y,T)$ is minimal, there exist $t_1,\cdots,t_k\in T$ such that $\bigcup_{s=1}^kt_sV=Y$.
By the continuity of $t_1,\cdots,t_k\in T$, there exists $\delta_r>0$ such that if $x,x'\in X$ with $d(x,x')\le \delta_r$, then
$$\max_{1\le s \le k} d(t_s^{-1}x,t_s^{-1}x')\le \delta.$$

Now for a given $y\in Y$, there is some $s(y)\in \{1,2,\cdots,k\}$ such that $t_{s(y)}^{-1}y\in V$. Then by \eqref{s-s2}, we can find  $x'_1\in \pi^{-1}(t_{s(y)}^{-1}y)\cap U_1,\ldots, x'_r\in \pi^{-1}(t_{s(y)}^{-1}y)\cap U_r$.
Moreover, by \eqref{s2} one has that
$$
\min_{1\le i\neq j\le r} d(x'_i, x'_j)> \d.
$$
Let $x_i=t_{s(y)}x'_i$ for $i=1,2,\cdots,r$.
Then $x_1,\cdots,x_r \in \pi^{-1}(y)$, and
by the choice of $\delta_r$, one has that $$
\min_{1\le i\neq j\le r} d(x_i, x_j)> \d_r.
$$
The proof is completed.
\end{proof}

For $T=\Z$, the following Proposition \ref{proximal-extension} is the consequence of \cite[Theorem 1.3]{Z} and \cite[Theorem 3.4]{Z}.
We will give a direct proof here for t.d.s. under abelian group actions. First recall some notions about subsets of $T$.

\medskip

A subset $S\subset T$ is {\em syndetic} if there exists a finite $F\subset T$ such that $FS=T$.
A subset $L\subset T$ is called {\em thick} if for every finite set $F\subset T$ one has that
$$L\cap \bigcap _{\gamma \in F}\gamma L\neq \emptyset.$$
Note that $S\subset T$ is syndetic if and only if $S\cap L\neq \emptyset$ for every thick set
$L\subset T$; and $L\subset T$ is thick if and only if $L\cap S\neq \emptyset$ for every syndetic set $S\subset T$.

\medskip

The following proposition is the key to improve Proposition \ref{MS-main}.
\begin{prop}\label{proximal-extension}
Let $\pi: (X,T)\rightarrow (Y,T)$ be an extension between minimal systems with $T$ abelian.
If $\pi$ is proximal but not almost one to one, then $(X,T)$ is multiple
$r$-sensitive for all $r\ge 2$. In particular, $h_\infty(X,T)=\infty$.
\end{prop}

\begin{proof}
Since $\pi$ is proximal but not almost one-to-one,
by Proposition \ref{almost N-1} $\pi$ is not almost finite to one.
Thus by Lemma \ref{lem-dist} for each fixed $r\ge 2$, there exists a constant $\d>0$ such that for all $y\in Y$,
there are $x_1,\ldots, x_r\in \pi^{-1}(y)$ with $$\min_{1\le i\neq j\le r} d(x_i, x_j)>3\d.$$

\medskip

\noindent {\bf Claim:} {\em For any non-empty open subset $U$ of $X$ and $F=\{t_1,t_2,\ldots, t_L\}\in T$
with $L\in \N$, there is some $m\in T$ such that $$mF=\{mt_1,mt_2,\ldots, mt_L\} \subseteq N(U,\d;r),$$ i.e. $N(U,\d;r)$ is thick.}


\begin{proof}[Proof of Claim.] Since $(X,T)$ is minimal, $\pi$ is semi-open. Hence for each
$i\in \{1,2,\ldots, L\}$, ${\rm int}(\pi(t_iU))\neq \emptyset$ and let
$y_i\in {\rm int}(\pi(T^iU))$. Choose points $z_{i1}, z_{i2},\ldots, z_{ir}\in \pi^{-1}(y_i)$ such that
\begin{equation*}
  \min_{1\le j\neq k\le r} d(z_{ij}, z_{ik})>3\d.
\end{equation*}
For $1\le i\le L$ and $1\le t\le r$, set
\begin{equation*}
  W_{it}=B(z_{it},\d)\cap \pi^{-1}({\rm int}(\pi(t_iU))),
\end{equation*}
where $B(x, a)=\{y\in X: d(x,y)<a\}$.

Since $(X,T)$ is minimal, the set of minimal points in $rL$ product system $(X^{rL},T)$ is dense.
Choose a minimal point $(p_{it})_{1\le i\le L\atop 1\le t\le r}$ in $\displaystyle \prod_{1\le i\le L\atop 1\le t\le r}W_{it}$.
Then by the definition of $W_{it}$, for each $1\le i\le L$ and $1\le t\le r$ there is some $x_{it}\in U$ such that
\begin{equation*}
  \pi(t_ix_{it})=\pi(p_{it}).
\end{equation*}
By Lemma \ref{clay}, in the product system $(X^{rL},T)$, points $(p_{it})_{1\le i\le L\atop 1\le t\le r}$
and $(t_ix_{it})_{1\le i\le L\atop 1\le t\le r}$ are proximal. Thus it follows from
Lemma \ref{Auslander} that there is some $m\in T$ such that
\begin{equation*}
  m\big((t_ix_{it})_{1\le i\le L\atop 1\le t\le r}\big)\in \prod_{1\le i\le L\atop 1\le t\le r}W_{it}.
\end{equation*}
That is,
\begin{equation*}
  mt_ix_{it}\in W_{it}, \ \forall i\in \{1,\ldots, L\}, t\in \{1,\ldots, r\}.
\end{equation*}
By the construction of $W_{it}$, one has that for all $i\in \{1,\ldots, L\}$
\begin{equation*}
  \min_{1\le j\neq k\le r} {\rm dist} (W_{ij},W_{ik}) >\d.
\end{equation*}
It follows that for all $i\in \{1,\ldots, L\}$
\begin{equation*}
  \min_{1\le j\neq k\le r} d (mt_ix_{ij},mt_ix_{ik}) >\d,
\end{equation*}
which means that $mF =\{mt_1,mt_2,\ldots, mt_L\} \subseteq N(U,\d;r)$.
\end{proof}

Now we show that $(X,T)$ is multiple $r$-sensitive. Let $U_1, U_2,\ldots, U_k$ be non-empty open
subsets of $X$. Fix a point $x_0\in X$. Since $(X,T)$ is minimal, there exist $t_1,t_2,\ldots, t_k\in T$
such that ${t_1}x_0\in U_1, {t_2}x_0\in U_2, \ldots, {t_k}x_0\in U_k$. Thus there is some open neighborhood $U$ of $x_0$ such that
\begin{equation*}
  {t_1}U\subseteq U_1,{t_2}U\in U_2,\ldots, {t_k}U\subseteq U_k.
\end{equation*}
By Claim, $N(U,\d;r)$ is thick, there is some $m\in T$ such that
\begin{equation*}
  \{mt_1,mt_2,\ldots, mt_k\}\subseteq N(U,\d;r).
\end{equation*}
It is easy to verify that
\begin{equation*}
  m\in \bigcap_{i=1}^kN(U_i,\d;r).
\end{equation*}
Thus $(X,T)$ is multiple $r$-sensitive. The proof is completed.
\end{proof}

For $n \geq 2$ one
writes $(X^n,T)$ for the $n$-fold product system $(X\times
\cdots \times X,T)$.

\begin{lem}\label{goodman} \cite[Proposition 2.4]{Go}
For a t.d.s. $(X,T)$ with $T$ abelian,
and any sequence $A\subset T$ we have
$$h_A(X^n,T)=nh_A(X,T)$$ for any $n\in\N$.
\end{lem}

\begin{lem}\label{Goodman1}\cite[Proposition 2.5]{Go}
Let $\pi: (X,T)\rightarrow (Y,T)$ be an extension with $T$ abelian and $A\subset T$ be a sequence. If $\pi$ is at most $N$ finite to one, i.e. $|\pi^{-1}(y)|\le N$ for all $y\in Y$, then $h_A(X,T)\le h_A(Y,T)+\log N$.
\end{lem}

Now we are ready to show prove Theorem \ref{Main-bd}.

\begin{proof}[Proof of Theorem \ref{Main-bd}]
First by  Proposition \ref{MS-main}, we have that $(X,T)$ has the following diagram:
\begin{equation*}
  \xymatrix{
  X \ar[d]_{\pi}
                & X' \ar[d]^{\pi'} \ar[l]_{\sigma'} \\
  Y=X_{eq}
                & Y'   \ar[l]_{\tau'}          }
\end{equation*}
where $\sigma'$ and $\tau'$ are proximal extensions, $\pi'$ is a
$N$ to one extension for some $N\in \N$ and $\pi$ is the maximal
equicontinuous factor.
Now we show that $\tau'$ and $\sigma'$ is almost one-to-one, and hence $\pi$ is almost finite to one.

\medskip

Assume that $\tau'$ is not almost one-to-one, then by Proposition \ref{proximal-extension}, $h_\infty(Y')=\infty$.
By the construction of diagram in Proposition \ref{MS-main} (See \cite{MS} for details), every point of
$Y'\subseteq 2^X$ consists of $N$ distinct elements of $X$.
Let $M_N(X)=\{A\in 2^X: |A|\le N\}$. Then $M_N(X)$ is a closed subset of $2^X$.
It is clear that $Y'\subset M_N(X)$ and thus $h_\infty(M_N(X),T)=\infty $.

Define $p:X^N\rightarrow M_N(X)$ such that $p\left((x_1,\ldots,x_N)\right)=\{x_1,\ldots,x_N\}$. We have
the following commuting diagram

\begin{equation*}
  \xymatrix{
  X^N \ar[d]_{p} \ar[r]^{T}
                 & X^N \ar[d]^{p} \\
  M_N(X) \ar[r]^{T}
                   & M_N(X)  }
\end{equation*}

This implies that $h_\infty(X^N, T)=\infty$, a contradiction by Lemma \ref{goodman}.
Thus, $\tau'$ is almost one to one and $h_\infty(Y')<\infty$.

Since $\pi'$ is finite to one, by Lemma \ref{Goodman1} $h_\infty(X' )\le h_\infty(Y')+\log N<\infty$. By Proposition \ref{proximal-extension},
$\sigma'$ is also almost one to one. Thus $\pi$ is almost finite to one.
The proof is completed.
\end{proof}





A minimal system $(X,T)$ is called {\em pointed distal} if there is some point $x_0\in X$, the only point
proximal to $x_0$ is itself. By Veech's structure for pointed distal systems and Theorem \ref{Main-bd},
any bounded minimal system under abelian group action is pointed distal. Thus we have

\begin{cor}
Let $(X,T)$ be a minimal system with $T$ abelian. If $(X,T)$ is not pointed distal, then $h_\infty(X,T)=\infty$.
\end{cor}


\bigskip

\appendix
\section{The proof of Proposition \ref{independent sets}}
The proof of  Proposition \ref{independent sets} basically is similar to the one of
\cite[Proposition 3.3.]{FGJO}. For the completeness, we include a proof.

From now on, $k$ is a fixed natural number with $k\ge2$.
Let $\Sigma_{n}=\{1,\ldots,k\}^{n}$ and $\Sigma_{*}=\bigcup_{n\in\mathbb{N}}\Sigma_{n}$.
Denote by $|a|$ the length of a word $a\in\Sigma_{*}$. Let $H$
be a locally compact second countable Hausdorff topological group
with left Haar measure $\Theta_H$. By Birkhoff-Kakutani theorem there exists a left invariant metric $d$ (see \cite{Struble} for example). Let
$e$ be the unit element.

Let $\Theta_{H}^{r}$ be the right Haar measure on $H$. If $C\subset H$
is a comapct set with positive measure and we set
\[
\eta^{C}(\epsilon)=\frac{\Theta_{H}^{r}(B_{\epsilon}(C))}{\Theta_{H}^{r}(C)}-1,
\]
where $B_\ep(C)=\{x\in H: d(x,c)<\ep\}$. Since $\Theta$ is regular,
$\lim_{\epsilon\rightarrow0}\eta^{C}(\epsilon)=0$.

\medskip

When $k=2$, the following lemma is Lemma 3.5. in \cite{FGJO}.

\begin{lem}\label{lem:tech lemma for finite k}
Suppose that $C\subset H$ is
a compact set with $\Theta_{H}^{r}(C)>0$ and $\{\xi_{a}\}_{a\in\sum_{*}}$
is a family of elements $\xi_{a}\in H$. Let $\{\epsilon_{n}\}_{n\in\mathbb{N}}$
be a sequence of positive real numbers such that
\[
\epsilon_{n}\geq\sup_{a\in\sum_{n}}d(e,\xi_{a}).
\]
For $j\in\mathbb{N}$, $n\in\mathbb{N}\cup\{\infty\}$, let $\delta_{j}^{n}=\sum_{\ell=j}^{n}\epsilon_{\ell}$.
Further, given $n\in\mathbb{N}$ and $a\in\Sigma_{n}$, let $\gamma_{a}=\xi_{a_1}
\xi_{a_1a_2}\ldots\xi_{a_1a_2\ldots a_n}=\prod_{j=1}^{n}\xi_{a_{1},\ldots,a_{j}}$.
Then for each $n\in\mathbb{N}$, we have

\begin{equation}
\Theta_{H}^{r}\left(\bigcap_{a\in\Sigma_{n}}C\gamma_{a}^{-1}\right)
\geq\Theta_{H}^{r}(C)\left(1-\sum_{j=1}^{n}k^{j-1}(k-1)\eta^{C}(\delta_{j}^{n})\right).\label{eq: ineq for Haar}
\end{equation}
\end{lem}

\begin{proof}
We prove the lemma by induction on $n$.
When $n=1$, note that $\{\xi_{a}\}_{|a|=1}=\{\xi_{i}\}_{1\leq i\leq k}$. For $1\le i\le k$, one has that
$$d(e, \gamma_{i}^{-1})=d(e,\gamma_{i})=d(e,\xi_i)\leq\epsilon_{1}.$$
Thus $\bigcup_{i=1}^k C\gamma_{i}^{-1}\subset B_{\delta_{1}^{1}}(C)$, where $\d_1^1=\ep_1$.
As $\Theta_{H}^{r}(B_{\delta_{1}^{1}}(C))=(1+\eta^{C}(\delta_{1}^{1}))\Theta_{H}^{r}(C)$, one has that
\[
\Theta_{H}^{r}\left(B_{\delta_{1}^{1}}(C)\setminus C\gamma_{i}^{-1}\right)=\eta^{C}(\delta_{1}^{1})\Theta_{H}^{r}(C).
\]
Thus one has that
\[
\Theta_{H}^{r}\left(B_{\delta_{1}^{1}}(C)\setminus\left(\bigcap_{i=1}^{k}C\gamma_{i}^{-1}\right)\right)
=\Theta_{H}^{r}\left(\bigcup_{i=1}^{k}\left(B_{\delta_{1}^{1}}(C)
\setminus C\gamma_{i}^{-1}\right)\right)\leq k\eta^{C}(\delta_{1}^{1})\Theta_{H}^{r}(C),
\]
and therefore
\begin{equation*}
\begin{split}
 \Theta_{H}^{r}\left(\bigcap_{i=1}^{k}\left(C\gamma_{i}^{-1}\right)\right)&
  \geq\Theta_{H}^{r}\left(B_{\delta_{1}^{1}}(C)\right)-k\eta^{C}(\delta_{1}^{1})\Theta_{H}^{r}(C)
 \\
 &= \Theta_{H}^{r}(C)\left(1-(k-1)\eta^{C}(\delta_{1}^{1})\right).
\end{split}
\end{equation*}
The base case is done.

\medskip

Now suppose that Equation \eqref{eq: ineq for Haar} holds for $n$,
all sets $C\subset H$ and all collection $\{\xi_{a}\}_{a\in\sum_{*}}$
and $\{\epsilon_{n}\}_{n\in\mathbb{N}}$ as above. Let $\xi_{a}^{(i)}=\xi_{ia}$ and
$\gamma_{a}^{(i)}=\xi_i^{-1}\cdot \gamma_{ia}$ for all $1\leq i\leq k$ and $a\in \Sigma_*$. Notice
that
\[
\{\gamma_{a}\}_{a\in\Sigma_{n+1}}=\{\xi_{i}\gamma_{a}^{(i)}\}_{1\leq i\leq k,a\in\Sigma_{n}}.
\]
Thus it holds that
\[
\bigcap_{a\in\Sigma_{n+1}}C\gamma_{a}^{-1}=\bigcap_{1\leq i\leq k}\left(\left(\bigcap_{a\in
\Sigma_{n}}C(\gamma_{a}^{(i)})^{-1}\right)\xi_{i}^{-1}\right).
\]
Let $\displaystyle J_{i}=\left(\bigcap_{a\in\Sigma_{n}}C(\gamma_{a}^{(i)})^{-1}\right)$.
By the induction hypothesis, one has that
\[
\Theta_{H}^{r}\left(J_{i}\xi_{i}^{-1}\right)=\Theta_{H}^{r}\left(J_{i}\right)
\geq\Theta_{H}^{r}(C)\left(1-\sum_{j=1}^{n}k^{j-1}(k-1)\eta^{C}(\delta_{j+1}^{n+1})\right).
\]
Notice that $J_{i}\xi_{i}^{-1}\subset B_{\delta_{1}^{n+1}}(C)$ for all $1\le i\le k$. Therefore, one has that
\begin{equation*}
\begin{split}
  &\hskip0.5cm  \Theta_{H}^{r}\left(\bigcap_{a\in\Sigma_{n+1}}C\gamma_{a}^{-1}\right) =\Theta_{H}^{r}
   \left({\displaystyle \bigcap_{1\leq i\leq k}}\left(J_{i}\xi_{i}^{-1}\right)\right)\\
  & =\Theta_{H}^{r}(B_{\delta_{1}^{n+1}}(C))-\Theta_{H}^{r}\left(B_{\delta_{1}^{n+1}}(C)
   \setminus\left({\displaystyle \bigcap_{1\leq i\leq k}}J_{i}\xi_{i}^{-1}\right)\right)\\
 & =\Theta_{H}^{r}(B_{\delta_{1}^{n+1}}(C))-\Theta_{H}^{r}\left({\displaystyle
 \bigcup_{1\leq i\leq k}}\left(B_{\delta_{1}^{n+1}}(C)\setminus J_{i}\xi_{i}^{-1}\right)\right)\\
 & \geq\Theta_{H}^{r}(B_{\delta_{1}^{n+1}}(C))-{\displaystyle \sum_{i=1}^{k}}\Theta_{H}^{r}
 \left(B_{\delta_{1}^{n+1}}(C)\setminus J_{i}\xi_{i}^{-1}\right)\\
 & =-(k-1)\Theta_{H}^{r}(B_{\delta_{1}^{n+1}}(C))+{\displaystyle \sum_{i=1}^{k}}
 \left(\Theta_{H}^{r}(J_{i}\xi_{i}^{-1})\right)\\
 & \geq\Theta_{H}^{r}(C)\left(k\left(1-\sum_{j=1}^{n}k^{j-1}(k-1)\eta^{C}
 (\delta_{j+1}^{n+1})\right)-(k-1)(1+\eta^{C}(\delta_{1}^{n+1}))\right)\\
 & =\Theta_{H}^{r}(C)\left(1-\sum_{j=1}^{n+1}k^{j-1}(k-1)\eta^{C}(\delta_{j}^{n+1})\right).
\end{split}
\end{equation*}
The proof is completed.
\end{proof}

 Given metric spaces $X$ and $H$, a continuous map $\beta:X\rightarrow H$ is called almost one-to-one if
 $X_0 = \{x\in :\beta^{-1}(\{\beta(x)\}) = \{x\}\}$ is dense in $X$. Points in $X_0$ are called injectivity points of $\beta$.
 Recall that a compact subset $W$ of a metric space  $X$ is called {\it proper} if $\overline{int(W)} = W$.

\begin{lem}\cite[Lemma 2.4]{FGJO} \label{lem:proper one to one}
Suppose that $X,H$ are metric
spaces and $\beta:X\rightarrow H$ is an almost one-to-one continuous map. Then images
of proper subsets of $X$ under $\beta$ are proper subsets of $H$.
\end{lem}

\medskip

Now we give the proof of Proposition \ref{independent sets}.

\begin{proof}[Proof of Proposition \ref{independent sets}]
As $\mathcal{G}$ is a residual set, assume that $\mathcal{G}=\bigcap_{n\in\mathbb{N}}G_{n}$,
where each $G_{n}$ is an open and dense subset of $H$. We will construct
an infinite sequence of point $\{t_{n}\}_{n\in\mathbb{N}}$ in $T$, a sequence $\{r(n)\}_{n\in \N}$ of positive real
numbers, a collection $\{\gamma_{a}\}_{a\in\Sigma_{*}}$ in $H$ and a collection $\{U_{a}\}_{a\in\Sigma_{*}}$ of compact subsets of $H$
such that for all $n\in \N$ and $a=a_1a_2\ldots a_n \in \Sigma_n$,

\medskip

\text{($I_n$)}\quad  $U_{a}=\overline{B_{r(n)}(\gamma_{a})}\subset(t_{n}^{-1}\text{int}(V_{a_n}))\cap G_{n}$;

\medskip

\text{($II_n$)}\quad  $\bigcup_{i=1}^k U_{ai}\subset U_{a}$.

\medskip

\text{($III_n$)}\quad  $\sup_{a\in \Sigma_n} d(e,\xi_a)\le \ep_n$, where $\{\xi_{a}\}_{a\in\sum_{*}}$
is a family of elements $\xi_{a}\in H$ defined by  $\gamma_{a}=\xi_{a_1}\xi_{a_1a_2}\ldots\xi_{a_1a_2\ldots a_n}=\prod_{j=1}^{n}\xi_{a_{1},\ldots,a_{j}}$.

\medskip

Once we have ($I_n$),($II_n$) for all $n\in \N$, then it will prove the statement. Firstly, we are to see that $t_i\neq t_j$ for any $i\neq j\in \mathbb{N}$. In fact, given $i\neq j\in \mathbb{N}$. We take $n=\max\{i,j\}$ and $a=a_1a_2\cdots a_n\in \Sigma_n$ with $a_i\neq a_j$. Then by $I_n$ one has $t_i(U_a)\subset \text{int}(V_{a_i})$ and $t_j(U_a)\subset \text{int}(V_{a_j})$.
Combing this with the fact $\text{int}(V_{a_i})\cap \text{int}(V_{a_j})=\emptyset$ (see (ii)),
one has $t_i\neq t_j$.

Next let $I=\{t_n:n\in \N\}$. Then $I$ is an infinite set of $T$. For $a\in \{1,2,\ldots,k\}^I$ and $n\in \N$, Let $a^{(n)}=a_{t_1}a_{t_2}\ldots a_{t_n}$.
By ($II_n$), $\bigcap _{n\in \N} U_{a^{(n)}}$ is a nested intersection of compact sets and therefore non-empty.
By ($I_n$), any $h\in \bigcap_{n\in \N} U_{a^{(n)}}$ has the property that
$h\in \G$ and $t_nh\in {\rm int} (V_{a_{t_n}})$ for all $n\in \N$. Thus one has (\ref{eq: in the int}).

\medskip

We will construct $\{t_{n}\}_{n\in\mathbb{N}}$, $\{r(n)\}_{n\in \N}$, $\{\gamma_{a}\}_{a\in\Sigma_{*}}$ and $\{U_{a}\}_{a\in\Sigma_{*}}$
by induction on $|a|=n$. Let $$C=\bigcap_{1\leq i\leq k}V_{i}.$$ As
$\Theta_{H}$ and $\Theta_{H}^{r}$ are mutually absolutely continuous,
one has that $\Theta_{H}^{r}(C)>0$. We fix a sequence $\{\epsilon_{n}\}_{n\in\mathbb{N}}$ of positive real numbers
such that
\[
\sum_{j=1}^{\infty}k^{j-1}(k-1)\eta^{C}(\delta_{j}^{\infty})<1,
\]
where $\delta_{j}^{n}$ are defined as in Lemma \ref{lem:tech lemma for finite k}.

\medskip

First we show the case $n=1$. Since $T$ is a dense subgroup, for any $\epsilon>0$,
$\bigcup_{t\in T}tB_{\epsilon}(e)=H$. Choose $t_{1}\in T$ such that
$\left(t_{1}B_{\epsilon_{1}}(e)\right)\cap C\neq\emptyset$. As $C\subset V_{i}$
for each $1\leq i\leq k$, one has that
$$\left(t_{1}B_{\epsilon_{1}}(e)\right)\cap V_{i}\neq\emptyset.$$
As $\overline{\text{int}(V_{i})}=V_{i}$, one has that
$$\left(t_{1}B_{\epsilon_{1}}(e)\right)\cap\text{int}(V_{i})\neq\emptyset \ \text{for any}\ 1\leq i\leq k.$$

Since $G_{1}$ is an open dense subset, it follows for $1\leq i\leq k$,
$$G_{1}\cap B_{\epsilon_{1}}(e) \cap t^{-1}_1{\rm int}(V_{i})\neq\emptyset$$
is a non-empty open set.
Now for $1\leq i\leq k$, choose $\gamma_{i}\in H$ and $r(1)>0$ such that
$$U_{i}=\overline{B_{r(1)}(\gamma_{i})}\subset G_{1}\cap B_{\epsilon_{1}}(e) \cap t^{-1}_1{\rm int}(V_{i}).$$
Let $\xi_{a}=\gamma_{a}$ for $a\in\{1,\ldots,k\}$,
one has that
$$\sup_{a\in\Sigma_{1}} d(e,\xi_{a})\leq\epsilon_{1}.$$
Thus the base case is done.

\medskip

Suppose now that $\{t_{i}\}_{i=1}^n$, $\{r(i)\}_{i=1}^n$, $\{\gamma_{a}\}_{a\in\cup_{i=1}^n\Sigma_{i}}$
and $\{U_{a}\}_{a\in\cup_{i=1}^n\Sigma_{i}}$ have been chosen
and satisfy ($I_n$), ($II_n$) and ($III_n$). By
Lemma \ref{lem:tech lemma for finite k} and ($III_n$), one has that
\begin{equation*}
\begin{split}
\Theta_{H}^{r}(\bigcap_{a\in\Sigma_{n}}C\gamma_{a}^{-1}) & \geq \Theta_{H}^{r}(C)\left(1-\sum_{j=1}^{n}k^{j-1}(k-1)\eta^{C}(\delta_{j}^{n})\right) \\
 &\geq \Theta_{H}^{r}(C)\left(1-\sum_{j=1}^{n}k^{j-1}(k-1)\eta^{C}(\delta_{j}^{\infty})\right)>0.
\end{split}
 \end{equation*}
In particular, $\bigcap_{a\in\Sigma_{n}}C\gamma_{a}^{-1}\neq \emptyset$, and choose
$h\in\bigcap_{a\in\Sigma_{n}}C\gamma_{a}^{-1}$. Thus $\gamma_a\in h^{-1}C$
for all $a\in\Sigma_{n}$. Choose $t_{n+1}\in T$ close enough to $h$ and $r'(n+1)<r(n)$ such that
\[
t_{n+1}^{-1}C\cap B_{r'(n+1)}(\gamma_{a})\neq\emptyset \quad \text{ for all} \ a\in\Sigma_{n}.
\]
Since $C=\bigcap_{1\leq i\leq k}V_{i}$ and $V_i=\overline{{\rm int} (V_i)}$ for $1\leq i\leq k$, one has that
$$\left(t_{n+1}^{-1}\text{int}(V_{i})\right)\cap B_{r'(n+1)}(\gamma_{a})\neq \emptyset, \quad 1\le i\le k$$
are non-empty open sets. As $G_{n+1}$ is open and dense, it follows that
\[
\left(t_{n+1}^{-1}\text{int}(V_{i})\right)\cap B_{r'(n+1)}(\gamma_{a})\cap G_{n+1}\neq \emptyset, \quad 1\le i\le k
\]
are non-empty open sets.
Now for each $a\in\Sigma_{n},i\in\{1,\ldots,k\}$, there exists $\gamma_{ai}\in H$ and $r(n+1)>0$
such that
\[
U_{ai}=\overline{B_{r(n+1)}(\gamma_{ai})}\subset\left(t_{n+1}^{-1}\text{int}(V_{i})\right)\cap B_{r'(n+1)}(\gamma_{a})\cap G_{n+1}.
\]
Notice that $$U_{ai}\subset B_{r'(n+1)}(\gamma_{a})\subset\overline{B_{r(n)}(\gamma_{a})}=U_{a},$$
and one has ($I_{n+1}$) and ($II_{n+1}$).
Recall that $\xi_{ai}=\gamma_{a}^{-1}\gamma_{ai}$ for all $a\in\Sigma_{n},i\in\{1,\ldots,k\}$.
One has that for all $a\in\Sigma_{n},i\in\{1,\ldots,k\}$,
$$d(e, \xi_{ai})=d(e, \gamma_a^{-1}\gamma_{ai})=d(\gamma_a,\gamma_{ai})\le r'(n+1)\le \ep_{n+1},$$
which is ($III_n$).
The proof is completed.
\end{proof}


\end{document}